\documentclass[11pt,a4paper]{article}
\usepackage{indentfirst}
\usepackage{fancyhdr}
\usepackage{color}
\usepackage{mathrsfs}
\usepackage{amsfonts,amssymb,amsmath,amsthm}
\allowdisplaybreaks[2]
\usepackage{authblk}
\usepackage[unicode,pdftex]{hyperref}
\hypersetup{
	colorlinks = true,
	citecolor = green,
	anchorcolor = blue,
	linkcolor = blue}
\usepackage{enumitem}

\newcommand\keywords[1]{\textbf{Keywords}: #1}
\newcommand\msc[1]{\textbf{MSC}: #1}

\newtheorem{thm}{Theorem}[section]
\newtheorem{lemma}[thm]{Lemma}

\newtheorem{coro}[thm]{Corollary}
\newtheorem{prop}[thm]{Proposition}

\newtheoremstyle{rem}{10pt}{10pt}{\rmfamily}{}{\bfseries}{.}{.5em}{}
\theoremstyle{rem}
\newtheorem{rem}[thm]{Remark}

\numberwithin{equation}{section}

\title{The well-posedness of stochastic Korteweg--de Vries equations  revisited}
\author[1]{Jie Chen}
\author[2]{Fan Gu\footnote{Corresponding author. E-mail address: gufan@amss.ac.cn}}
\affil[1]{\scriptsize \textit{School of Science, Jimei University, Xiamen 361021, P.R. China}}
\affil[2]{\scriptsize \textit{School of Statistics and Mathematics, Central University of Finance and Economics,  Beijing 102206, P.R. China}}
\date{}

\begin{document}
	\maketitle	
	\begin{abstract}
		In this paper, we propose a new view, which leads to  almost sure well-posedness in $H^{s}(\mathbb{R}), s\geq 0$, for studying stochastic KdV equations. Different from \cite{de1999white} or \cite{kdvmuti},  by introducing a solution space inspired by \cite{guo2009global}, we prove the local  well-posedness result only under natural $H^s(\mathbb{R}), s\geq 0$ conditions parallel to  deterministic KdV equations. Furthermore, just basing on the $L_x^2$ conservation law of KdV equations, we extend the solution to  a global one. The well-posedness frame obtained  in this paper  not only reduces several restrictions of the noise kernel, but may also have crucial values when one deals with  dynamical problems of stochastic KdV equations. 
	\end{abstract}	

\msc{60H15, 35Q53}

\keywords{Stochastic KdV equations, Bilinear estimate, Well-posedness}
	
	\section{Introduction}\label{sec:introduction}
	We study the local and global well-posedness of the stochastic Korteweg--de Vries equation (s-KdV) in $H^s(\mathbb{R})$, $s\geq 0$. The model is defined on $(\Omega, \mathcal{F},\mathbb{P},\{\mathcal{F}_t\}_{t\geq 0} )$ and driven by the additive noise
	\begin{equation}\label{askdv}
		\left\{
		\begin{aligned}
			&d u=-(u_{xx}+u^2)_xdt +\Phi dW_t,\\
			&u|_{t = 0} = u_0\in H^s(\mathbb{R}),
		\end{aligned}
		\right.
	\end{equation}
	or the multiplicative noise
	\begin{equation}\label{mskdv}
		\left\{
		\begin{aligned}
			&d u=-(u_{xx}+u^2)_xdt +u\Phi dW_t,\\
			&u|_{t = 0} = u_0\in H^s(\mathbb{R}),
		\end{aligned}
		\right.
	\end{equation}
	where $u$ is real-valued, $\Phi$ is an operator defined by
	\begin{equation}\label{kernel}
		\Phi f(x) = \int_{\mathbb{R}} k(x,y)f(y)~dy,\quad \forall~f\in L^2(\mathbb{R}).
	\end{equation}
	$dW_t/dt$ adapted to $\{\mathcal{F}_t\}_{t\geq0}$ is a space-time independent white noise on $L^2(\mathbb{R})$. $W_t$ can be represented as $\sum_{j=0}^{\infty}\beta_j(t)e_j$, where $\{\beta_j\}$ is a sequence of mutually independent standard $1d$ Brownian motions and  $\{e_j\}$ is an orthonormal basis of $L^2(\mathbb{R})$.
	
	The well-posedness of s-KdV had been studied in \cite{KdVH1,de1999white,kdvmuti}. In \cite{de1999white}, de Bouard--Debussche--Tsutsumi estimate the stochastic convolution in Bourgain spaces $X^{s,b}$, $b<1/2$. However, there is no bilinear estimate in this space due to \textit{high $\times$ low $\rightarrow$ high} interaction in the frequency space. Then, they assume the initial data $u_0\in \dot{H}^{-3/8}$ to make sure that this interaction is very weak. In \cite{kdvmuti}, de Bouard--Debussche use $v(t) = u(t)-e^{-\partial_{xxx}}u_0$ as a new variable. As a result, while they no longer need to assume $u_0\in \dot{H}^{-3/8}$, their argument requires $L^1(\mathbb{R})$ restriction of the noise kernel, when they control the stochastic convolution.
	
	In this paper, only under some regularity conditions parallel to the deterministic case, we develop a new approach to study the well-posedness of stochastic KdV equations. 
	
	The article begin with the bilinear estimate. Inspired by Guo \cite{guo2009global}, we introduce a new solution space $Z^{s,b}_T$ containing the norm $L_x^2L_T^\infty$ for the low frequency part (For the definition of $Z^{s,b}_T$, see \eqref{defixsbt}. Note that $Z^{s,b}_T\hookrightarrow C([0,T];H^s)$ trivially). This norm suffices to handle the bilinear estimate for the \textit{high $\times$ low $\rightarrow$ high} frequency interaction.
	
	As a result, we need to establish the low-frequency estimate for the stochastic convolution in $L_\omega^2L_x^2L_T^\infty$. The difficulty of the $L_\omega^2L_x^2L_T^\infty$ estimate, owing to taking $L_T^\infty$ first, has been observed in the existing literature; see for example \cite{KdVH1}. Thus, when \cite{KdVH1} tried to take  $L^2_\omega$ first, the estimate may not be sharp. To overcome this, we obtain sharp estimates for the operator $P_1e^{-t\partial_{xxx}}$ by exploiting its low-frequency properties and employing Taylor expansion.
	
	Problems also arise in the  extension of the local solution. Just like \cite{kdvmuti}, we first prove an $L^2_\omega L_T^\infty L_x^2$ priori estimate. However, different from \cite{kdvmuti}, we cannot get a $\|\cdot\|_{L^{2}_\omega Z_T^{0,b}}$ estimate of the solution by the $L^2_\omega L_T^\infty L_x^2$ priori estimate in a pathwise view because of only $L_\omega^{2m}L_x^2L_T^\infty, m=1$ estimate can be done for the stochastic convolution. Fortunately, by contradiction, we can show the global well-posedness in the almost surely sense inspired by the proof of the deterministic case.

	The well-posedness structure of stochastic KdV equations proposed in this paper not only can help us further figure out the connection between deterministic KdV equations and stochastic KdV equations, but also has the following advantages:
	First, we get rid of several restrictions of the noise kernel (such as convolution type operators and the $L^1_x$ restriction). Second, the well-posedness results expand to $H_x^s, s\geq 0$ only relying on the $L_x^2$ conservation law of the deterministic KdV equation. Last and most importantly, it provides a new approach to study stochastic KdV equations, which sometimes might be important in dealing with dynamical properties of stochastic KdV equations such as \cite{averaging_kdv}.
	
	In \cite{averaging_kdv},  we need to deal with a trilinear estimate of a coupling term like $f^2g$, where $f\in L_t^rL_x^2, r\in[2,\infty)$.  If we use the $\|\cdot\|_{Y_{b,0,-q}}, q>0$ type norm to treat it, by duality, we should do estimates like
	$$
	\begin{aligned}
		(h,f^2g)
		\lesssim&\|D^qh\|_{L^4_{t}L_x^\infty}\|D^{-q}(f^2g)\|_{L^{4/3}_{t}L_x^1}.
	\end{aligned}
	$$
	However, there is no general Sobolev embedding that can deal with $\|D^{-q}(\cdot)\|_{L_x^1}$. Here,  $ \|\cdot\|_{Y^T_{b,0,-q}}$ is defined by $\|D^{-q}(\cdot)\|_{X^{0,b}_T}+\|\cdot\|_{X^{0,b}_T}$.
	
	Now, we list our main results in this paper.
	\begin{thm}\label{globalwella}
		Let $s\geq 0$, $k\in L^2_{y}H^s_x$, $3/8 < b < 1/2$. Then for any $u_0\in H^s(\mathbb{R})$, there exists a unique global solution $u$ of \eqref{askdv} $s.t.$ $u\in Z^{s,b}_T$ for any $T>0$ almost surely.
	\end{thm}
	
	\begin{thm}\label{globalwellm}
		Let $s\geq 0$, $J^s_xk\in L_x^\infty L^2_y$, $3/8< b < 1/2$. Then for any $u_0\in H^s(\mathbb{R})$, there exists a unique global solution $u$ of \eqref{mskdv} $s.t.$  $u\in Z^{s,b}_T$ for any $T>0$ almost surely.
	\end{thm}

	\begin{rem}
		Note that for $k(x,y) = K(x-y)$, $\|J^s_xk\|_{L_x^\infty L_y^2} = \|J^s_xK\|_{L^2_x} = \|K\|_{H^s}$. Thus we remove the $L^1$ restriction and generalize the main result in \cite{kdvmuti}.
	\end{rem}
	
	\begin{rem}
		The well-posedness of the stochastic KdV on the torus had also been extensively studied. See for example \cite{de2005periodic,oh2009invariance,oh2010periodic,oh2024global} and reference therein. In fact, by subtracting the average, the resonant function for KdV on the torus has better property than KdV on the real line. By using this, estimating the second iteration and approximation, Oh showed the well-posedness of the stochastic KdV on the torus with additive space-time white noise. However, to the best of the authors' knowledge, a similar result on the real line appears to remain open.
	\end{rem}

	Since Theorem \ref{globalwella} can be obtained by modifying the proof of Theorem \ref{globalwellm} slightly, we would omit the proof of it totally. In fact, By the estimate in Proposition \ref{stochesti} for the stochastic convolution term, one can use a fixed point argument on each path to show the existence of solution locally. Then by using the priori estimate on the $L^2_x$ norm of the solution, one can extend the local solution to a global one.

	This paper is organized in the following manner: In Section \ref{sec:preliminaries}, we introduce some definitions and notations. We prove some estimates for the bilinear term and the stochastic convolution in Section \ref{estimate}. Then in Section \ref{wellposednesstrun}, we show the global well-posedness of the truncated equation. Finally, we prove the main result in Section \ref{global_solution}. In Appendices \ref{truncated}, \ref{proofoffrac}, we give the proof of two technical lemmas. In Appendix \ref{recallmu}, we recall some multilinear estimates for the KdV equation in Bourgain spaces.
	
	\section{Preliminaries}\label{sec:preliminaries}
	Let $a,b\in \mathbb{R}$. $a\lesssim b$ means $a\leq Cb$ for a constant $C>0$. If $a\lesssim b$ and $b\lesssim a$, we denote this by $a \thicksim b$. We use $\langle x \rangle$ to denote $(1+|x|^2)^{1/2}$.
	
	For any $ \varphi \in \mathcal{S}'(\mathbb{R}\times\mathbb{R})$, we use $\hat{\varphi}(\tau,\xi)$ or $ \mathscr{F}_{t,x}\varphi(\tau,\xi)$ to represent the space-time Fourier transform of $\varphi$.	We use $ \mathscr{F}_x\varphi(t,\xi)$ to denote the space Fourier transform of $\varphi$.
	
	Let $J^s =\mathscr{F}^{-1}_{\xi}\langle\xi\rangle^s\mathscr{F}$, $D^s = \mathscr{F}^{-1}_{\xi} |\xi|^s \mathscr{F}$,
	$$\|u\|_{H^s(\mathbb{R})}:=\|J^su\|_{L^2(\mathbb{R})}, \quad\|u\|_{\dot{H}^s(\mathbb{R})}:=\|D^su\|_{L^2(\mathbb{R})}.$$
	For any $f\in L^2(\mathbb{R})$, we denote the frequency cut-off operator by
	\begin{equation}\label{freque}
		P_1f=\mathscr{F}^{-1}_\xi (\chi_{[-1,1]}\mathscr{F}_xf),\quad P_Nf = \mathscr{F}^{-1}(\chi_{[-N,N]\setminus [-N/2,N/2]}\mathscr{F}_xf),~N\geq 2
	\end{equation}
	and $P_{>1} = I-P_1$.

    Let $\|u\|_{L_t^qL_x^r}:= \|\|u(t,x)\|_{L_x^r}\|_{L_t^q}$, $L_T^q L_x^r$ be the restricted norm of $L_t^qL_x^r$ on $[0,T]\times \mathbb{R}$. Let $\|u\|_{L_x^rC_T}:=\|\|u(t,x)\|_{C_t([0,T])}\|_{L^r_x}$. Similarly, we define $L_x^rL_T^q$ and $C_TH^s_x$. 
    
    We denote the Bourgain norm by
    $$\|u\|_{X^{s,b}} := \|\langle\xi\rangle^s\langle\tau-\xi^3\rangle^b\hat{u}(\tau,\xi)\|_{L^2_{\tau,\xi}} = \|w(t,\xi)\|_{L_\xi^2H_t^b}$$
    where $w(t,\xi)=\langle\xi\rangle^se^{-it\xi^3}\mathscr{F}_x(u(t))(\xi)$. Instead of the classical restricted norm for $X^{s,b}$ on $[0,T]$, we define
	\begin{equation*}
		\|u\|_{X^{s,b}_T}^2=\int_{\mathbb{R}}\int_0^T \left((1+t^{-2b})|w(t,\xi)|^2+\int_0^t\frac{|w(t,\xi)-w(t',\xi)|^2}{|t-t'|^{1+2b}}~dt'\right)dtd\xi.
	\end{equation*}
	To demonstrate the rationality of the definition for $X^{s,b}_T$, we need the following lemma. 
	
	\begin{lemma}\label{lem:equa-norm}
		Let $0\leq b<1/2$. Then
		\begin{equation*}
			c\|u\|_{X^{s,b}_T}\leq\inf\{\|\tilde{u}\|_{X^{s,b}}:\tilde{u}|_{[0,T]\times\mathbb{R}} = u\}\leq C\|u\|_{X^{s,b}_T},
		\end{equation*}
		where $c$ and $C$ depend on $b$ only.
	\end{lemma}
	\begin{proof}[\textbf{Proof}]
		We only need to show:
		\begin{align*}
			&\quad\int_0^T \left((1+t^{-2b})|f(t)|^2+\int_0^t\frac{|f(t)-f(t')|^2}{|t-t'|^{1+2b}}~dt'\right)dt\\
			&\sim \inf\{\|\tilde{f}\|_{H^b}^2:\tilde{f}|_{[0,T]} = f\}.
		\end{align*}
		By Lemma 2.1 in \cite{kdvmuti} we have  
		\begin{align*}
			&\quad\inf\{\|\tilde{f}\|_{H^b}^2:\tilde{f}|_{[0,T]} = f\}\\
			& \sim \int_{0}^{T}  \left((1+t^{-2b}+(T-t)^{-2b})|f(t)|^2+\int_{0}^{t}\frac{|f(t)-f(t')|^2}{|t-t'|^{1+2b}}~dt'\right)dt.
		\end{align*}
		Thus it is enough to obtain
		\begin{equation*}
			\int_{0}^{T}  (T-t)^{-2b}|f(t)|^2~dt\lesssim \int_0^T \left(T^{-2b}|f(t)|^2+\int_0^t\frac{|f(t)-f(t')|^2}{|t-t'|^{1+2b}}~dt'\right)dt.
		\end{equation*}
		By the scaling argument we only need to consider the case $T = 1$. Then by Lemma 2.1 in \cite{kdvmuti} and the inner description of Sobolev space we have
		\begin{align*}
			\int_{0}^1  (1-t)^{-2b}|f(t)|^2~dt &\lesssim \inf\{\|\tilde{f}\|^2_{H^b}:\tilde{f}|_{[0,1]} = f\}\\
			&\sim \int_0^1 \left(|f(t)|^2+\int_0^t\frac{|f(t)-f(t')|^2}{|t-t'|^{1+2b}}~dt'\right)dt.
		\end{align*}
		See for example \cite{triebel2010theory}, pages 208--209. We conclude the proof.
	\end{proof}
    
    In the proof of the local well-posedness for \eqref{mskdv}, we use the norm 
    \begin{equation}\label{defixsbt}
    	\|u\|_{Z^{s,b}_T}:=\|P_{1}u\|_{L_x^2C_T}+\|u\|_{X_T^{s,b}}+\|u\|_{C_TH_x^s}.
    \end{equation}
	\begin{rem}
		Let $u_0\in H^s$. It is easy to see that $e^{-t\partial_{xxx}}u_0\in Z^{s,b}_T$, $\forall~T>0$. 
		For any $u\in Z^{s,b}_T$, $\|u\|_{Z_t^{s,b}}$ is a continuous, non-decreasing function for $t\in [0,T]$.
	\end{rem}
	
	By the Duhamel formula, the mild solution of \eqref{mskdv} is
	\begin{equation}\label{mild}
		u(t)=U(t)u_0-\int_{0}^{t} U(t-t')(u(t')^2)_x~dt'+\int_{0}^{t}U(t-t') (u(t')\Phi dW_{t'}),
	\end{equation}
	where $U(t) = e^{-t\partial_{xxx}}$.
	
	Let $\theta\in C_0^\infty(\mathbb{R})$ be a radial, positive function with $\text{supp}\ \theta\subset[-2,2]$, $\theta(t) = 1$, $\forall~ t\in[-1,1]$.
	
	To prove the local well-posedness of $\eqref{mskdv}$, following the argument in \cite{kdvmuti} we consider the truncated equation:
	\begin{equation*}\tag{$2.7_R$}\label{truncated_eq}
		u(t)=U(t)u_0-\int_{0}^{t} U(t-t')(B_R(u)(t'))_x~dt'+\int_{0}^{t}U(t-t') (u(t')\Phi dW_{t'})
	\end{equation*}
	where
	$$B_R(u)(t) = \left(\theta(\|P_1u\|_{L_x^2C_t}/R)P_1u(t)+\theta(\|u\|^2_{X^{0,b}_t}/R^2)P_{>1}u(t)\right)^2.$$
	
	For any $K>0$, $u\in L^2_\omega Z_{T}^{0,b}$ adapted to the filtration $\{\mathcal{F}_t\}_{t\geq 0}$, we define the bounded stopping time
	\begin{equation*}
		\sigma^T_K(u)=\left\{
		\begin{aligned}
			&\mathrm{inf}\{t> 0: \|u\|_{Z^{0,b}_t}\geq K\}, & \|u\|_{Z^{0,b}_T}\geq K,\\
			&T, & \|u\|_{Z^{0,b}_T} < K,
		\end{aligned}
		\right.\quad \mathrm{a.s.}~\mathbb{P}.
	\end{equation*}
		
	\section{Linear and bilinear estimates}\label{estimate}
	In this section, we show some linear and bilinear estimates. Firstly we recall the classical estimates in Bourgain space. By Lemmas 3.11--3.12 in \cite{erdogan_tzirakis_2016}, we have
	\begin{lemma}\label{inhomoesti}
		Let $s\in \mathbb{R}$, $0<\epsilon<1$, $0<T<1$. Then
		\begin{equation}\label{restricted_norm_esti_1}
			\left\|\int_{0}^{t} U(t-t') f(t')  ~dt'\right\|_{X_T^{s,(1+\epsilon)/2}}\lesssim T^{\epsilon/2}\|f\|_{X_T^{s,\epsilon-1/2}}.
		\end{equation}
	\end{lemma}

	In the spirit of Lemma 2.2 in \cite{kdvmuti}, we have the following nonlinear estimate.
	\begin{lemma}\label{truncated_equation}
		Let $s\geq 0$, $b\in[0,1/2)$. For any $R>0$, $T>0$, $u, v\in X^{s,b}_T$, we have
		$$ 
		\|\theta(\|u\|^2_{X^{0,b}_t}/R^2)u\|_{X^{s,b}_T}\lesssim \|u\|_{X^{s,b}_T},\ 	\|\theta(\|u\|^2_{X^{0,b}_t}/R^2)u\|_{X^{0,b}_T}\lesssim R,
		$$
		$$
		\|\theta(\|u\|^2_{X^{0,b}_t}/R^2)u-\theta(\|v\|^2_{X^{0,b}_t}/R^2)v\|_{X^{0,b}_T}\lesssim \|u-v\|_{X^{0,b}_T}.
		$$
	\end{lemma}
	We would prove this lemma in Appendix \ref{truncated}.
	\begin{rem}\label{rem:why_wrong}		
		It seems that the argument for proving Lemma 2.2 in \cite{kdvmuti} only works for $0\leq b<1/4$. In fact let $v(t) = \eta(t)U(t)u_0$, $\eta|_{[-1,1]} = 1$, $0\neq u_0\in C_0^\infty$, $\mathscr{F}_x u_0(0) = 0$, $w(t) = U(-t)v(t) = \eta(t)u_0$. Then $\|\chi_{[r,t]}v\|_{X^{s,b}}\sim \|\chi_{[r,t]}\|_{H^b}\sim (t-r)^{1/2-b}$, $0<r<t<1$. Thus for any $T>0$ one has
		\begin{align*}
			\int_{\mathbb{R}}\int_0^T\int_0^t \|\chi_{[r,t]}v\|_{X^{s,b}}^2\frac{|\mathscr{F}_xw(t,\xi)|^2}{|t-r|^{1+2b}}~drdtd\xi\gtrsim \int_0^{\min\{T,1\}}\int_0^t |t-r|^{-4b}~drdt.
		\end{align*}
		For $b\geq 1/4$, the right hand is infinite. Thus the control of term $II$ in their proof is not complete.
	\end{rem}

	\subsection{Bilinear estimates}
	We have the following bilinear estimates which would be useful in the proof of local well-posedness. For pioneering works, see \cite{bourgain1993fourier}, \cite{kenig1996bilinear}, Proposition 2.2 in \cite{de1999white}, and \cite{guo2009global}.
	
	\begin{lemma}\label{bilinearestforkdv}
		Let $1/4< a, b < 1/2$, $a+2b\geq 5/4$, and $0<T<1$. We have
		\begin{equation*}
			\begin{aligned}
				\|(uv)_x\|_{X_T^{0,-a}}&\lesssim \|P_1 u\|_{L_T^\infty L_x^2}\|P_1v\|_{L_{T}^2L_x^\infty}+ \|P_1 u\|_{L_x^2L_T^\infty}\|P_{>1}v\|_{X_T^{0,b}} \\
				&\quad +\|P_{>1}u\|_{X_T^{0,b}}\|P_1 v\|_{L_x^2L_T^\infty}+ \|P_{>1}u\|_{X_T^{0,b}}\|P_{>1}v\|_{X_T^{0,b}},
			\end{aligned}
		\end{equation*}
		which implies
		\begin{equation*}
			\|(uv)_x\|_{X_T^{0,-a}}\lesssim \|u\|_{Z_T^{0,b}}\|v\|_{X_T^{0,b}}+\|u\|_{X_T^{0,b}}\|v\|_{Z_T^{0,b}}.
		\end{equation*}
	\end{lemma}
	\begin{proof}[\textbf{Proof}]
		We only need to show
		\begin{align}
			\|(P_1uP_1v)_x\|_{X_T^{0,-a}} &\lesssim \|P_1 u\|_{L_T^\infty L_x^2}\|P_1v\|_{L_{T}^2L_x^\infty}, \label{lowlow}\\
			\|(P_1uP_{>1}v)_x\|_{X_T^{0,-a}} &\lesssim \|P_1 u\|_{ L_x^2L_T^\infty}\|P_{>1}v\|_{X_T^{0,b}},\label{lowhigh}\\
			\|(P_{>1}uP_{>1}v)_x\|_{X_T^{0,-a}} &\lesssim \|P_{>1} u\|_{X^{0,b}_T}\|P_{>1}v\|_{X^{0,b}_T}.\label{highhigh}
		\end{align}
		For \eqref{lowlow}, by the H\"{o}lder inequality we have
		\begin{align*}
			\|(P_1uP_1v)_x\|_{X_T^{0,-a}}&\lesssim \|(P_1uP_1v)_x\|_{L_T^2L^2_x}\lesssim \|P_1uP_1v\|_{L_T^2L^2_x}\\
			&\lesssim \|P_1u\|_{L_T^\infty L_x^2}\|P_1v\|_{L_T^2L_x^\infty}.
		\end{align*}
		By the extension and the duality, for \eqref{lowhigh}--\eqref{highhigh} we only need to show
		\begin{align}
			\left|\int_{\mathbb{R}^2}(P_1uP_{>1}v)_xw~dxdt\right|&\lesssim \|P_1u\|_{L_x^2L_t^\infty}\|v\|_{X^{0,b}}\|w\|_{X^{0,a}},\label{duallowhigh}\\
			\left|\int_{\mathbb{R}^2}(P_{>1}uP_{>1}v)_xw~dxdt\right|&\lesssim \|u\|_{X^{0,b}}\|v\|_{X^{0,b}}\|w\|_{X^{0,a}}.\label{dualhighhigh}
		\end{align}
		By the Bernstein inequality, one has $\|\partial_xP_1 u\|_{L_x^2L_t^\infty}\lesssim \|P_1u\|_{L_x^2L_t^\infty}$. Then by the H\"{o}lder inequality and the local smoothing estimate, for any $\epsilon>0$ we have
		\begin{equation}\label{lowhighv}
			\begin{aligned}
				&\quad\left|\int_{\mathbb{R}^2}(P_1uP_{>1}v)_xw~dxdt\right|\\
				&\leq \|(P_1uP_{>1}v)_x\|_{L_{t,x}^2}\|w\|_{L^2_{t,x}}\\
				&\lesssim (\|P_1u\|_{L^2_xL^\infty_t}\|\partial_xP_{>1}v\|_{L^\infty_xL^2_t}+\|\partial_xP_1 v\|_{L_x^2L_t^\infty}\|P_{>1}v\|_{L_x^\infty L_t^2})\|w\|_{X^{0,0}}\\
				&\lesssim\|P_1u\|_{L_x^2L_t^\infty}\|P_{>1}v\|_{X^{0,1/2+\epsilon}}\|w\|_{X^{0,0}}\\
				&\lesssim\|P_1u\|_{L_x^2L_t^\infty}\|v\|_{X^{0,1/2+\epsilon}}\|w\|_{X^{0,0}}.
			\end{aligned}			
		\end{equation}
		On the other hand, we also have
		\begin{equation}\label{P1P>1_2}
			\begin{aligned}
				\left|\int_{\mathbb{R}^2} (P_1uP_{>1}v)_x w~dxdt\right|& = \left|\int_{\mathbb{R}^2} (P_1uP_{>1}v) \partial_xw~dxdt\right|\\
				&\leq \|P_1uP_{>1}v\|_{L_x^1 L_t^2}\|\partial_xw\|_{L_x^\infty L_t^2}\\
				&\lesssim \|P_1 u\|_{L_x^2L_t^\infty}\|P_{>1}v\|_{L_{t,x}^2} \|w\|_{X^{0,1/2+\epsilon}}\\
				&\lesssim \|P_1 u\|_{L_x^2L_t^\infty}\|v\|_{X^{0,0}} \|w\|_{X^{0,1/2+\epsilon}}.
			\end{aligned}
		\end{equation}
		Then for any $a,b>1/4$, combining \eqref{lowhighv} and \eqref{P1P>1_2}, by the interpolation \cite{calderon} we obtain \eqref{duallowhigh}.
		
		By the refined Strichartz estimate $X^{0,1/3}\hookrightarrow L_{t,x}^4$ (c.f. Proposition 6.4 \cite{tao2001multilinear} for the result on the torus) and the interpolation, we have $X^{0,2/9}\hookrightarrow L_{t,x}^3$. See for example Theorem 3.18 \cite{erdogan_tzirakis_2016}. By the H\"{o}lder inequality, for $a,b\geq 2/9$ one has
		\begin{align*}
			\left|\int_{\mathbb{R}^2}(P_{>1}uP_{>1}v)_xP_1w~dxdt\right|&\leq \|P_{>1}u\|_{L_{t,x}^3}\|P_{>1}v\|_{L_{t,x}^3}\|\partial_xP_1w\|_{L_{t,x}^3}\\
			&\lesssim \|u\|_{X^{0,b}}\|v\|_{X^{0,b}}\|w\|_{X^{0,a}}.
		\end{align*}
		Let $Q_1u = \mathscr{F}^{-1}_{t,x}(\chi_{[-1,1]}(\tau-\xi^3)\hat{u}(\tau,\xi))$ and
		\begin{equation}\label{modulationdeco}
			Q_Lu = \mathscr{F}^{-1}_{t,x}(\chi_{[-L,L]\setminus [-L/2,L/2]}(\tau-\xi^3)\hat{u}(\tau,\xi)),\quad L\in 2^{\mathbb{N}}.
		\end{equation}
		We use $a_{\min}$, $a_{\mathrm{med}}$, $a_{\max}$ to denote the minimum, medium, maximum of $a_1$, $a_2$, $a_3$. Note that for $L_{\max}\leq N_1N_2N_3/10$ or $N_{\mathrm{med}}\leq N_{\max}/4$ one has
		$$\int_{\mathbb{R}^2}(Q_{L_1}P_{N_1}uQ_{L_2}P_{N_2}v)_xQ_{L_3}P_{N_3}w~dxdt = 0.$$
		Now we would assume $N_{\min}\geq 2$, $L_{\max}\sim \max\{N_1N_2N_3,L_{\mathrm{med}}\}$, $N_{\mathrm{med}}\sim N_{\max}$. 
		
		If $N_{\min}\sim N_{\max}$, by Lemma \ref{multiesti} $(c)$, we have
		\begin{align*}
			&\quad\left|\int_{\mathbb{R}^2}(Q_{L_1}P_{N_1}uQ_{L_2}P_{N_2}v)_xQ_{L_3}P_{N_3}w~dxdt\right|\\
			&\lesssim N_3L_{\min}^{1/2}L_{\mathrm{med}}^{1/4}N_{\max}^{-1/4}\|Q_{L_1}P_{N_1}u\|_{L^2_{t,x}}\|Q_{L_2}P_{N_2}v\|_{L^2_{t,x}}\|Q_{L_3}P_{N_3}w\|_{L^2_{t,x}}\\
			&\lesssim L_{\max}^{1/4}L_{\mathrm{med}}^{1/4}L_{\min}^{1/2}(L_1L_2)^{-b}L_3^{-a}\|u\|_{X^{0,b}}\|v\|_{X^{0,b}}\|w\|_{X^{0,a}}.
		\end{align*}
		Then for $a,b>1/4$, $a+2b>1$ one has
		\begin{align*}
			&\quad\sum_{L_{\max}\gtrsim N_{\max}^3,N_{\min}\sim N_{\max}\geq 2}\left|\int_{\mathbb{R}^2}(Q_{L_1}P_{N_1}uQ_{L_2}P_{N_2}v)_xQ_{L_3}P_{N_3}w~dxdt\right|\\
			&\lesssim \sum_{L_{\max}\gtrsim N_{\max}^3,N_{\min}\sim N_{\max}}L_{\max}^{1/4}L_{\mathrm{med}}^{1/4}L_{\min}^{1/2}(L_1L_2)^{-b}L_3^{-a}\|u\|_{X^{0,b}}\|v\|_{X^{0,b}}\|w\|_{X^{0,a}}\\
			&\lesssim \|u\|_{X^{0,b}}\|v\|_{X^{0,b}}\|w\|_{X^{0,a}}.
		\end{align*}
		If $N_3 = N_{\min}\ll N_{\max}$, by Lemma \ref{multiesti} $(b)$, we have
		\begin{align*}
			&\quad\left|\int_{\mathbb{R}^2}(Q_{L_1}P_{N_1}uQ_{L_2}P_{N_2}v)_xQ_{L_3}P_{N_3}w~dxdt\right|\\
			&\lesssim N_3L_{\mathrm{med}}^{1/2}L_{\min}^{1/2}N_{\max}^{-1/2}N_3^{-1/2}\|Q_{L_1}P_{N_1}u\|_{L^2_{t,x}}\|Q_{L_2}P_{N_2}v\|_{L^2_{t,x}}\|Q_{L_3}P_{N_3}w\|_{L^2_{t,x}}\\
			&\lesssim L_{\mathrm{med}}^{1/2}L_{\min}^{1/2}(L_1L_2)^{-b}L_3^{-a}\|u\|_{X^{0,b}}\|v\|_{X^{0,b}}\|w\|_{X^{0,a}}.
		\end{align*}
		Then for $a,b>1/4$, $a+2b>1$ one has
		\begin{align*}
			&\quad\sum_{L_{\max}\gtrsim N_3N_1^2, N_3 = N_{\min}\geq 2}\left|\int_{\mathbb{R}^2}(Q_{L_1}P_{N_1}uQ_{L_2}P_{N_2}v)_xQ_{L_3}P_{N_3}w~dxdt\right|\\
			&\lesssim \sum_{L_{\max}\gtrsim N_3N_1^2, N_3 = N_{\min}}L_{\mathrm{med}}^{1/2}L_{\min}^{1/2}(L_1L_2)^{-b}L_3^{-a}\|u\|_{X^{0,b}}\|v\|_{X^{0,b}}\|w\|_{X^{0,a}}\\
			&\lesssim \|u\|_{X^{0,b}}\|v\|_{X^{0,b}}\|w\|_{X^{0,a}}.
		\end{align*}
		For $N_3\neq N_{\min}$, without loss of generality we would assume $N_1 = N_{\min}$. If $L_1 \neq
		L_{\max}$ or $L_{\mathrm{med}}\sim L_{\max}$, the same argument as for $N_3 = N_{min}$ also works. Thus we would assume $N_2\sim N_3\gg N_1$, $L_1\sim N_1N_2^2\gg \max\{L_2,L_3\}$. By using Lemma \ref{multiesti} $(b)$ again we obtain
		\begin{align*}
			&\quad\left|\int_{\mathbb{R}^2}(Q_{L_1}P_{N_1}uQ_{L_2}P_{N_2}v)_xQ_{L_3}P_{N_3}w~dxdt\right|\\
			&\lesssim N_3L_2^{1/2}L_3^{1/2}N_3^{-1/2}N_1^{-1/2}\|Q_{L_1}P_{N_1}u\|_{L^2_{t,x}}\|Q_{L_2}P_{N_2}v\|_{L^2_{t,x}}\|Q_{L_3}P_{N_3}w\|_{L^2_{t,x}}\\
			&\lesssim N_3^{1/2}N_1^{-1/2}L_2^{1/2}L_3^{1/2}(L_1L_2)^{-b}L_3^{-a}\|u\|_{X^{0,b}}\|P_{N_2}v\|_{X^{0,b}}\|P_{N_3}w\|_{X^{0,a}}.
		\end{align*}
		Thus for $a,b<1/2$, $1/2+2(1-2b-a)\leq 0$ which means $a+2b\geq 5/4$ we have
		\begin{align*}
			&\quad\sum_{2\leq N_1\ll N_2\sim N_3}\sum_{L_2,L_3\ll L_1\sim N_1N_2^2}\left|\int_{\mathbb{R}^2}(Q_{L_1}P_{N_1}uQ_{L_2}P_{N_2}v)_xQ_{L_3}P_{N_3}w~dxdt\right|\\
			&\lesssim \sum_{N_1\ll N_2\sim N_3}N_3^{1/2}N_1^{-1/2}(N_1N_2^2)^{1-2b-a}\|u\|_{X^{0,b}}\|P_{N_2}v\|_{X^{0,b}}\|P_{N_3}w\|_{X^{0,a}}\\
			&\lesssim\sum_{N_2\sim N_3}\|u\|_{X^{0,b}}\|P_{N_2}v\|_{X^{0,b}}\|P_{N_3}w\|_{X^{0,a}}\\
			&\lesssim \|u\|_{X^{0,b}}\|v\|_{X^{0,b}}\|w\|_{X^{0,a}}.
		\end{align*}
		We obtain \eqref{dualhighhigh} and conclude the proof.
	\end{proof} 

	By slightly modifying the proof of Lemma \ref{bilinearestforkdv}, we have
	\begin{coro}\label{generals}
		Let $s\geq 0$, $1/4< a, b < 1/2$, $a+2b\geq 5/4$, and $0<T<1$. We have
		\begin{equation*}
			\begin{aligned}
				\|(uv)_x\|_{X_T^{s,-a}}&\lesssim \|P_1 u\|_{L_x^2L_T^\infty}\|P_1v\|_{L_x^2L_T^\infty}\\
				&\quad + \|P_1 u\|_{L_x^2L_T^\infty}\|P_{>1}v\|_{X_T^{s,b}}+\|P_{>1}u\|_{X_T^{s,b}}\|P_1 v\|_{L_x^2L_T^\infty}\\
				&\quad + \|P_{>1}u\|_{X_T^{s,b}}\|P_{>1}v\|_{X_T^{0,b}}+ \|P_{>1}u\|_{X_T^{0,b}}\|P_{>1}v\|_{X_T^{s,b}}.
			\end{aligned}
		\end{equation*}
	\end{coro}
	For the low frequency, we have the following estimate.
	\begin{lemma}\label{lowfrebi}
		Let $0<T<1$. Then
		\begin{align*}
			&\quad\left\|P_1\int_{0}^{t} U(t-t')(u v)_x~d{t'}\right\|_{L_x^2 L^\infty_T}\\
			&\lesssim T^{1/2}(\|P_1u\|_{L_x^2L_T^\infty}+\|P_{>1}u\|_{X^{0,1/3}_T})(\|P_1v\|_{L_x^2L_T^\infty}+ \|P_{>1}v\|_{X^{0,1/3}_T}).
		\end{align*}
	\end{lemma}
	
	\begin{proof}[\textbf{Proof}]
		By the maximal function estimate and the H\"{o}lder inequality, one has
		\begin{align*}
			\left\|P_1\int_{0}^{t} U(t-t')(u v)_x~d{t'}\right\|_{L_x^2 L_T^\infty}			\lesssim \|P_1(uv)_x\|_{L_T^1L_x^2}\lesssim T^{1/2}\|u\|_{L_{T}^4L_x^4}\|v\|_{L_T^4L_x^4}.
		\end{align*}
		By Bernstein and Minkowski inequalities, we have  
		$$\|P_1u\|_{L_T^4L_x^4}\lesssim \|P_1u\|_{L_T^4L_x^2}\lesssim T^{1/4}\|P_1u\|_{L_T^\infty L_x^2}\lesssim \|P_1u\|_{L_x^2L_T^\infty}.$$
		Combining with $X^{0,1/3}\hookrightarrow L^4_{t,x}$, we conclude the proof.
	\end{proof}
	
	Combining Lemmas \ref{inhomoesti}, \ref{bilinearestforkdv}, \ref{lowfrebi}, and $X^{0,c}\hookrightarrow C_tL_x^2$ for $c>1/2$, by choosing $a = 1/2-\varepsilon$ for $0<\varepsilon<1/4$, $b<1/2$, and $a+2b>5/4$ we have
	\begin{prop}\label{biliearesti}
		Let $3/8<b<1/2$, $0<T<1$. Then for $\varepsilon<2b-3/4$, we have 
		\begin{equation*}
			\left\|\int_0^tU(t-t')(uv)_x~dt'\right\|_{Z_T^{0,b}}\lesssim T^{\varepsilon}\Big(\|u\|_{Z^{0,b}_T}\|v\|_{X^{0,b}_T}+\|u\|_{X^{0,b}_T}\|v\|_{Z^{0,b}_T}\Big).
		\end{equation*}
	\end{prop}

	Combining Lemma \ref{truncated_equation}, Corollary \ref{generals} with Lemma \ref{lowfrebi}, we also have the following proposition.
	\begin{prop}\label{trunbiliearestis}
		Let $s\geq 0$, $3/8<b<1/2$, $\varepsilon<2b-3/4$, $R>0$, $0<T<1$. Then, we have
		\begin{equation*}
			\left\|\int_0^tU(t-t')(B_R(u))_x~dt'\right\|_{Z_T^{s,b}}\lesssim RT^\varepsilon\|u\|_{Z_T^{s,b}}
		\end{equation*}
		and
		\begin{equation*}
		\left\|\int_0^tU(t-t')(B_R(u)-B_R(v))_x~dt'\right\|_{Z_T^{0,b}}\lesssim RT^{\varepsilon}\|u-v\|_{Z_T^{0,b}}.
		\end{equation*}
	\end{prop}
	
	\begin{rem}\label{rem:notprecise}
		Although Proposition \ref{biliearesti} has a concise expression, we still need to use the expressions in Lemmas \ref{bilinearestforkdv}, \ref{lowfrebi} in the proof of Proposition \ref{trunbiliearestis} and the local well-posedness.
	\end{rem}

	\subsection{Estimates of the stochastic convolution}
	In order to obtain control of the stochastic term by using the property of $k$, we need the following two lemmas. The first one is classical, and the second one is showed in Appendix \ref{proofoffrac}. Recall that $\Phi$ is the operator defined by \eqref{kernel} and $\{e_j\}$ is an orthonormal basis of $L^2_x$.
	\begin{lemma}\label{hilbertschmidt}
		For any $s\in \mathbb{R}$, we have $\|\Phi e_j\|_{l^2_j H^s_x} = \|k\|_{L_y^2H^s_x}$.
	\end{lemma}
	\begin{proof}[\textbf{Proof}]
		$\|\Phi e_j\|_{l^2_jH^s_x}$ is equal to the Hilbert--Schmidt norm of the operator with kernel $J^s_xk(x,y)$. Thus $\|\Phi e_j\|_{l^2_j H^s_x} = \|J^s_xk(x,y)\|_{L^2_{x,y}} =  \|k\|_{L_y^2H^s_x}$.
	\end{proof}	
	
	\begin{lemma}\label{fractionleib}
		Let $s\geq 0$, $u\in H^s_x$. We have 
		\begin{equation*}
			\| u(x) \Phi e_j\|_{l^2_j H^s_x} =\|u(x)k(x,y)\|_{L^2_y H^s_x}\lesssim \|u\|_{H^s_x}\|J^s_xk\|_{L_x^\infty L_y^2}.
		\end{equation*}
	\end{lemma}
\begin{proof}[\textbf{Proof}]
	We put the proof of this lemma in Appendix \ref{proofoffrac}.
\end{proof}
	
	Estimates of the stochastic convolution for the additive case and the multiplicative case are showed in the following lemma at the same time. In fact, $f_j$ would be $u\Phi e_j$ or $\Phi e_j$ when we use this lemma. For pioneering works, see for example Proposition 2.1 in \cite{de1999white} and Propositions 2.5, 2.7 in \cite{kdvmuti}.
	\begin{lemma}\label{stochasticesti}
		Let $s \in \mathbb{R}$, $0\leq 
		b<1/2$, $0<T<1$. For any $ j\in \mathbb{N}$, $f_j(t',x,\omega)$ is $\mathscr{B}\left([0,t]\times \mathbb{R}\right)\times \mathcal{F}_t$ measurable on $[0,t]\times \mathbb{R}\times\Omega  $. Then for any stopping time $\sigma$, we have 
		\begin{equation*}
			\left\|\sum_{j\in \mathbb{N}}\int_{0}^{t} U(t-t')f_j(t',x,\omega) ~d\beta_j(t')\right\|_{L_\omega^{2}Z^{s,b}_{T\wedge \sigma}} \lesssim \| f_j\|_{L_\omega^{2}L^2_{T\wedge \sigma}l^2_j H^s_x}.
		\end{equation*}
	\end{lemma}
	
	\begin{proof}[\textbf{Proof}]
		Firstly, by the unitary property of $U(t)$ on $H^s$ and the Burkholder--Davis--Gundy (BDG) inequality, it is easy to get   
		\begin{align*}
			&\quad\left\|\sum_j\int_{0}^{t} U(t-t')f_j(t',x,\omega) ~d\beta_j(t')\right\|_{L_\omega^{2}C([0,T\wedge\sigma];H^s_x)}\\
			& = \left\|\sum_j\int_{0}^{t} U(-t')f_j(t',x,\omega) ~d\beta_j(t')\right\|_{L_\omega^{2}C([0,T\wedge\sigma];H^s_x)}\\
			&\lesssim \|f_j\|_{L_\omega^{2}L_{T\wedge \sigma}^2l^2_jH^s_x}.
		\end{align*}
		For the low frequency part, by the Taylor expansion we have
		\begin{align*}
			P_1U(t)= \mathscr{F}^{-1}\chi_{(-1,1)}(\xi)e^{it\xi^3}\mathscr{F} & = \sum_{n=0}^\infty\frac{1}{n!}\mathscr{F}^{-1}\chi_{(-1,1)}(\xi)(it\xi^3)^n\mathscr{F}\\
			&  = \sum_{n=0}^\infty\frac{(-t)^n}{n!}\partial_x^{3n}P_1
		\end{align*}
		Then by the triangle inequality one has
		\begin{align*}
			&\quad \left\|P_1\sum_j\int_{0}^{t} U(t-t')f_j(t',x,\omega) ~d\beta_j(t')\right\|_{L_\omega^{2} L_x^2C_{T\wedge \sigma}}\\
			&=\left\|P_1U(t)\sum_j\int_{0}^{t} U(-t')f_j(t',x,\omega) ~d\beta_j(t')\right\|_{L^{2}_\omega L_x^2L_{T\wedge \sigma}^\infty}\\
			&\leq \sum_{n = 0}^\infty\frac{1}{n!}\left\|t^n\partial_x^{3n}P_1\sum_j\int_{0}^{t} U(-t')f_j(t',x,\omega) ~d\beta_j(t')\right\|_{L^{2}_\omega L_x^2 L_{T\wedge \sigma}^\infty}.
		\end{align*}
		Since $\|\partial_xP_1f\|_{L_x^2L_T^\infty}\leq C\|P_1f\|_{L_x^2L_T^\infty}$, we can control the upper term by
		\begin{align*}
			&\quad\sum_{n = 0}^\infty\frac{(CT)^n}{n!}\left\|P_1\sum_j\int_{0}^{t} U(-t')f_j(t',x,\omega) ~d\beta_j(t')\right\|_{L^2_\omega L_x^2 L_{T\wedge \sigma}^\infty}\\
			&\lesssim e^{CT} \left\|\sum_j\int_{0}^{t} P_1U(-t')f_j(t',x,\omega) ~d\beta_j(t')\right\|_{L_x^2 L^{2}_\omega L_{T\wedge \sigma}^\infty}.
		\end{align*}
		Then by the BDG inequality, we can control it by
		\begin{equation*}
			\|P_1U(-t')f_j(t,x,\omega)\|_{L_x^2 L^2_\omega L^2_{T\wedge\sigma}l^2_j}\lesssim \|f_j\|_{L_\omega^2L_{T\wedge \sigma}^2l^2_jH^s_x}.
		\end{equation*}
		By the definition of $X^{s,b}_T$, the It\^{o} isometry, and the Plancherel identity, one has
		\begin{align*}
			&\quad \left\|\sum_j\int_{0}^{t} U(t-t')f_j(t',x,\omega) ~d\beta_j(t')\right\|_{L_\omega^2X^{s,b}_{T\wedge \sigma}}^2\\
			&\lesssim \int_{\Omega}\int_{\mathbb{R}}\int_0^{T\wedge \sigma}(1+t^{-2b})\left|\sum_j\int_{0}^{t} \mathscr{F}_x(U(-t')J^s_xf_j(t',x,\omega)) ~d\beta_j(t')\right|^2~dtd\xi d\omega\\
			&\quad + \int_{\Omega}\int_{\mathbb{R}}\int_0^{T\wedge  \sigma}\int_0^t\left|\sum_j\int_{t'}^{t} \frac{\mathscr{F}_x(U(-\tilde{t})J^s_xf_j(\tilde{t},x,\omega))}{|t-t'|^{1/2+b}} ~d\beta_j(\tilde{t})\right|^2~dt' dtd\xi d\omega\\
			&= \int_{\Omega}\int_{\mathbb{R}}\int_0^{T\wedge \sigma}(1+t^{-2b})\sum_j\int_{0}^{t}|\mathscr{F}_x(U(-t')J^s_xf_j(t',x,\omega)) |^2~dt'dtd\xi d\omega\\
			&\quad+ \int_{\Omega}\int_{\mathbb{R}}\int_0^{T\wedge \sigma}\int_0^t\sum_j\int_{t'}^{t} \frac{|\mathscr{F}_x(U(-\tilde{t})J^s_xf_j(\tilde{t},x,\omega))|^2}{|t-t'|^{1+2b}} ~d\tilde{t}dt' dtd\xi d\omega\\
			&= \int_{\Omega}\int_{\mathbb{R}}\int_0^{T\wedge \sigma}(1+t^{-2b})\sum_j\int_{0}^{t} |J^s_xf_j(t',x,\omega) |^2~dt'dtdx d\omega\\
			&\quad + \int_{\Omega}\int_{\mathbb{R}}\int_0^{T\wedge \sigma}\int_0^t\sum_j\int_{t'}^{t} \frac{|J^s_xf_j(\tilde{t},x,\omega)|^2}{|t-t'|^{1+2b}} ~d\tilde{t}dt' dtdx d\omega\\
			&\lesssim \sum_j\int_{\Omega}\int_{\mathbb{R}}\int_0^{T\wedge \sigma}T^{1-2b} |J^s_xf_j(t',x,\omega) |^2~dt'dx d\omega\\
			&\lesssim T^{1-2b}\|f_j\|_{L^2_\omega L^2_{T\wedge \sigma }l^2_j H^s_x}^2.
		\end{align*}
		We conclude the proof due to $b<1/2$, $0<T<1$.
	\end{proof}	
	
	Combining Lemmas \ref{hilbertschmidt}--\ref{stochasticesti}, we have:
	\begin{prop}\label{stochesti}
		Let $0\leq b<1/2$, $0<T<1$, and $\Phi$ be the operator defined by \eqref{kernel}. Then for any stopping time $\sigma\geq 0$, we have 
		$$\left\|\int_{0}^{t} U(t-t')(\Phi dW_{t'})\right\|_{L_\omega^2Z^{s,b}_{T\wedge \sigma}}
		\lesssim T^{1/2}\|J^s_xk\|_{L^2_{x,y}},\quad \forall~s\in \mathbb{R}$$
		and
		$$
		\left\|\int_{0}^{t} U(t-t')(u\Phi dW_{t'})\right\|_{L_\omega^2Z^{s,b}_{T\wedge \sigma}}\lesssim T^{1/2}\|u\|_{L^2_\omega Z^{s,b}_{T\wedge\sigma}}\| J^s_xk\|_{L^\infty_x L^2_y},\quad \forall~s\geq 0.
		$$
	\end{prop}
	
	\section{The well-posedness of truncated equations}\label{wellposednesstrun}	
	In this section, we prove the global well-posedness of \eqref{truncated_eq}.
	\begin{prop}\label{localwellres}
		Fix $3/8<b< 1/2$ and let $R,T>0$, $s\geq 0$, $J^s_xk \in {L^\infty_x L^2_y}$. For any $u_0\in H^s$, there exists a unique solution $u$ of \eqref{truncated_eq} in $L^2_\omega Z^{s,b}_T$ and
		\begin{equation*}
			\|u\|_{L^2_\omega Z^{s,b}_T}\leq C(\|J^s_xk\|_{L_x^\infty L_y^2},R,T)\|u_0\|_{H^s}.
		\end{equation*}
	\end{prop}
	\begin{proof}[\textbf{Proof}]
		We firstly show $u\in L^2_\omega Z^{0,b}_{T_*}$ for some $0<T_*<1$ through a fixed point argument. Let
		\begin{align*}
			\mathcal{T}_R:u\mapsto
			U(t)u_0-\int_{0}^{t} U(t-t')(B_R(u))_x~dt'+\int_{0}^{t}U(t-t') (u\Phi dW_{t'}).
		\end{align*}
		By Propositions \ref{biliearesti}, \ref{trunbiliearestis}, and \ref{stochesti}, for $0<\varepsilon<2b-3/4$ (one may fix $\varepsilon = b-3/8$) we have
		\begin{align*}
			\|\mathcal{T}_R u\|_{L^2_\omega Z^{0,b}_{T_*}} &\lesssim \|u_0\|_{L^2}+T_*^{\varepsilon}R\|u\|_{L^2_\omega Z^{0,b}_{T_*}}+T_*^{1/2}\|k\|_{L_x^\infty L_y^2}\|u\|_{L^2_\omega Z^{0,b}_{T_*}}\\
			&\lesssim \|u_0\|_{L^2}+(T_*^{\varepsilon}R+T_*^{1/2}\|k\|_{L_x^\infty L_y^2})\|u_R\|_{L^2_\omega Z^{0,b}_{T_*}}
		\end{align*}
		and
		$$\|\mathcal{T}_R u- \mathcal{T}_R v\|_{L^2_\omega Z^{0,b}_{T_*}}\lesssim (T_*^{\varepsilon}R+T_*^{1/2}\|k\|_{L_x^\infty L_y^2})\|u-v\|_{L^2_\omega Z^{0,b}_{T_*}}.$$
		Thus, by choosing $T_* = T_*(R,\|k\|_{L_x^\infty L_y^2},\varepsilon)<1$ sufficiently small, we obtain that $\mathcal{T}_R$ is a strict contraction on $L^2_\omega Z^{0,b}_{T_*}$.   
		
		Since $T_*$ relies on $R$, $\|k\|_{L_x^\infty L_y^2}$ and $\varepsilon$ only, we can obtain the solution on $[0,T]$ by dividing $[0,T]$ into finite numbers of intervals shorter than $T_*$ and using contraction mapping argument on each subinterval. For example, we construct the solution on $(T_0,T_0+T_*]$, when we already have solution on $[0,T_0]$.  Consider the mapping on $\{v\in L^2_\omega Z^{0,b}_{T_*}: v(t)~\mbox{is}~\mathcal{F}_{t+T_0}\mbox{-adapted}, v(0) = u(T_0)\}$:
		\begin{equation*}
			\tilde{\mathcal{T}}_R:v\mapsto U(t)u(T_0)-\int_{0}^{t} U(t-t')(\tilde{B}_R(v))_x~dt'+\int_{0}^{t}U(t-t') (v\Phi dW_{t'+T_0}),
		\end{equation*}
		where $v(t)$ is $\mathcal{F}_{t+T_0}$-adapted and 
		$$\tilde{B}_R(v)(t) = \left(\theta(\|P_1\tilde{v}\|_{L_x^2C_{T_0+t}}/R)P_1v(t)+\theta(\|\tilde{v}\|_{X^{s,b}_{T_0+t}}^2/R^2)P_{>1}v(t)\right)^2$$ 
		with
		\begin{align*}
			\tilde{v}(t) = u(t),~0\leq t\leq T_0;\quad\tilde{v}(t) = v(t),~T_0<t\leq T_*.
		\end{align*}
		Similar to the proof of Proposition \ref{trunbiliearestis}, we also have 
		$$
		\left\|\int_{0}^{t} U(t-t')(\tilde{B}_R(v)-\tilde{B}_R(w))_x~dt'\right\|_{Z^{0,b}_{T_*}}\lesssim T_*^{\varepsilon}R\|v-w\|_{Z^{0,b}_{T_*}}.$$
		Then one can get the contractility of $\tilde{\mathcal{T}}_R$ on $L^2_\omega Z^{0,b}_{T_*}$, which means that there exists a unique $\mathcal{F}_{t+T_0}$ adapted solution $v$ of $\tilde{\mathcal{T}}_Rv = v$. Let $u(t+T_0) = v(t)$, $\forall~t\in [0,T_*]$. Then we obtain a solution of \eqref{truncated_eq} on $[0,T_0+T_*]$. Thus by the iteration we have a unique solution of $u\in L^2_\omega Z_T^{0,b}$ for any $0<T<\infty$ and
		$$\|u\|_{L^2_\omega Z_T^{0,b}}\leq C(\|k\|_{L_x^\infty L_y^2},R,T)\|u_0\|_{L^2}.$$		
		
		If $u_0\in H_x^s$, $\|J^s_xk\|_{L_x^\infty L_y^2}<\infty$, by Lemmas \ref{truncated_equation}, Corollary \ref{generals}, Proposition \ref{stochesti}, for any $n\geq 1$ we have
		$$\|\mathcal{T}^{n}_R(0)\|_{L^2_\omega Z_{T_*}^{s,b}}\lesssim \|u_0\|_{H^s_x}+(T_*^{\varepsilon}R+T_*^{1/2}\|J^s_xk\|_{L_x^\infty L_y^2})\|\mathcal{T}^{n-1}_R(0)\|_{L^2_\omega  Z_{T_*}^{s,b}}.$$
		Thus, we can choose $T_*=T_*(R,\varepsilon, \|J^s_xk\|_{L_x^\infty L_y^2})$ sufficiently small such that 
		$$\|\mathcal{T}_R^n(0)\|_{L^2_\omega  Z_{T_*}^{s,b}}\lesssim \|u_0\|_{H_x^s},\quad \forall~n.$$ 
		Since the solution $u$ of \eqref{truncated_eq} is the limit of $\mathcal{T}^n_R(0)$ in $L^2_\omega  Z_{T_*}^{0,b}$ as $n\rightarrow\infty$, we conclude that the
		solution $u\in L^2_\omega  Z_{T_*}^{s,b}$ and $\|u\|_{L^2_\omega  Z_{T_*}^{s,b}}\lesssim \|u_0\|_{H_x^s}$. Finally, by dividing $[0,T]$ into small intervals and the iteration as before, we finish the proof.
	\end{proof}
	
	\section{The proof of Theorem \ref{globalwellm}}\label{global_solution}
	For $u_0\in L^2$, let $u_R$ be the solution of \eqref{truncated_eq}. Since for any $T>0$, $R_2>R_1>0$, we have $u_{R_2}(t)=u_{R_1}(t)$, $t\in [0, \sigma_{R_1}^T(u_{R_1})]$, $ \text{a.s.}$   $\mathbb{P}$, we would define $u(t) = u_R(t)$, $\forall~t\leq \sigma_R^T(u_R)$. To get the global well-posedness of \eqref{mskdv}, we need to show 
	\begin{equation}\label{global_aim}
		\sigma^T:=\lim_{R\uparrow \infty} \sigma^T_R(u_R) =T, \quad \mathrm{a.s.}~ \mathbb{P}.
	\end{equation}
	To prove \eqref{global_aim}, we first establish a priori estimate for $\|u\|_{L^{2}_\omega L_{\sigma^T}^\infty L_x^2}$. Let
	\begin{align*}
		M(u) = \int_{\mathbb{R}}u^2 ~dx.
	\end{align*}

	Firstly we show the priori estimate for solutions with high regularity.
	\begin{lemma}\label{lem:regularity_priori_estimate}
		Let $s\geq 3$. Suppose $u_0\in H_x^s$ and $ J^s_xk\in L_x^\infty L_y^2$. For any fixed $ R,T>0$, let $u_R$ be the solution constructed in Proposition \ref{localwellres}.  Then for any stopping time $0\leq \sigma\leq \sigma_R^T(u_R)$ almost surely, we have 
		\begin{equation}\label{L_x^2_priori_esti}
			\|u_R\|_{L_\omega^{p}L_{\sigma}^\infty L_x^2}\leq C(p,T,\|k\|_{L_x^\infty L_y^2})\|u_0\|_{L_x^2}, \quad \forall~2\leq p<\infty.
		\end{equation}
	\end{lemma}
	\begin{proof}[\textbf{Proof}]
		By the It\^{o} formula, we  have  
		\begin{equation}\label{ito_regularity}
			\begin{aligned}
				&\quad M(u_R(t\wedge\sigma))\\
				& = M(u_0) +\int_0^{t\wedge\sigma} \mathrm{tr}(\Phi^*u_R^2\Phi)~dt' +2\int_0^{t\wedge\sigma}(u_R,u_R\Phi dW_{t'}).
			\end{aligned}
		\end{equation}
		By the BDG inequality, the  H\"{o}lder inequality, and Lemma \ref{fractionleib}, for $2\leq 
		p<\infty$ we have
		\begin{align*}
			&\quad\|M(u_R)\|_{L_\omega^{p/2}L_{t\wedge \sigma}^\infty}\\
			&\leq M(u_0)+\|\|u_R\Phi e_j\|_{l^2_jL^2_x}^2\|_{L^{p/2}_\omega L^1_{t\wedge\sigma}}+2C(p)\|\|u_R^2\Phi e_j\|_{l^2_j L^1_x}\|_{L_\omega^{p/2} L^2_{t\wedge\sigma}}\\
			&\leq M(u_0)+\|M(u_R)\|k\|_{L_x^\infty L_y^2}^2\|_{L_\omega^{p/2}L^1_{t\wedge\sigma}}+2C(p)\|M(u_R)\|k\|_{L_x^\infty L_y^2}\|_{L_\omega^{p/2} L^2_{t\wedge\sigma}}.
		\end{align*}
		By the Cauchy--Schwarz inequality one has
		\begin{align*}
			&\quad 2C(p)\|M(u_R)\|k\|_{L_x^\infty L_y^2}\|_{L_\omega^{p/2} L^2_{t\wedge\sigma}}\\
			&\leq 2C(p)\|k\|_{L_x^\infty L_y^2}\|M(u_R)\|_{L_\omega^{p/2} L^1_{t\wedge\sigma}}^{1/2}\|M(u_R)\|_{L_\omega^{p/2} L^\infty_{t\wedge\sigma}}^{1/2}\\
			&\leq \frac{1}{2}\|M(u_R)\|_{L_\omega^{p/2} L^\infty_{t\wedge\sigma}}+2C(p)^2\|k\|^2_{L_x^\infty L_y^2}\|M(u_R)\|_{L_\omega^{p/2} L^1_{t\wedge\sigma}}.
		\end{align*}
		Thus 
		$$\|M(u_R)\|_{L_\omega^{p/2}L_{t\wedge \sigma}^\infty}\\
		\leq 2M(u_0)+(2+4C(p)^2)\|k\|_{L_x^\infty L_y^2}^2\|M(u_R)\|_{L_\omega^{p/2} L^1_{t\wedge\sigma}}.$$
		Then by the Gronwall inequality we have
		\begin{align*}
			\|M(u_R)\|_{L_\omega^{p/2}L_{\sigma}^\infty} \leq 2e^{(2+4C(p)^2)\|k\|_{L_x^\infty L_y^2}^2T}M(u_0).
		\end{align*}
		We complete the proof.
	\end{proof}
	
	Now we would approximate general $k$ and initial data $u_0$ by high regularity functions and show the priori estimate for general solutions. We first establish a technical lemma.
	
	\begin{lemma}\label{lem:convergence_need_to_explain}
		Let $f(\omega,x)\in L_\omega^2L_x^2$ and $k(x,y)\in L_x^\infty L_y^2 $. Then, we have 
		\begin{equation}\label{convergence_result}
			\lim_{n\uparrow\infty}\mathbb{E}\left\|f(k-k_n) \right\|_{L_{x,y}^2}^2=0.
		\end{equation}
	\end{lemma}
	\begin{proof}
		For the sake of simplicity, let us introduce following notations
		$$
		k_{>n}:=k-k_n:=k-\tilde{P}_nk:=\mathscr{F}^{-1}_\xi[\theta_n(\xi)\mathscr{F}k(\xi)](x),
		$$
		$$
		k_{a<\cdot<b}:=\tilde{P}_bk-\tilde{P}_ak =k_b-k_a,
		$$
		for any $0<a,b<\infty$.
		
		Note that 
		\begin{equation}\label{auxi_inequality_1}
			\begin{aligned}
				\left\|f(k-k_n) \right\|_{L_{x,y}^2}
				\leq&\|f\|_{L_x^2}\left(\|k\|_{L_x^\infty L_y^2}+\|\mathscr{F}^{-1}_\xi[\theta_n]\ast k\|_{L_x^\infty L_y^2}\right)\\
				\leq&\|f\|_{L_x^2}\left(\|k\|_{L_x^\infty L_y^2}+\|\mathscr{F}^{-1}_\xi[\theta_n]\|_{L_x^1}\|k\|_{L_x^\infty L_y^2}\right)\\
				\lesssim&\|f\|_{L_x^2}\|k\|_{L_x^\infty L_y^2}.
			\end{aligned}
		\end{equation}
		Thus, according to \eqref{auxi_inequality_1} and the dominated convergence theorem, we can prove \eqref{convergence_result} by illustrating
		\begin{equation}\label{auxi_lim_1}
			\begin{aligned}
				\lim_{n\uparrow\infty}\|fk_{>n}\|_{L_{x,y}^2}=0, \ \text{a.s.}~ \mathbb{P}.
			\end{aligned}
		\end{equation}
		
		We will illustrate \eqref{auxi_lim_1} basing on the following partition in the frequency space:
		\begin{equation}\label{partition}
			\begin{aligned}
				\|fk_{>n}\|_{L_{x,y}^2}\leq& \|\tilde{P}_{>n/10}(fk)\|_{L_{x,y}^2}+\|\tilde{P}_{>n/10}(fk_n)\|_{L_{x,y}^2}+\|\tilde{P}_{n/10}(f_{>n/2}k_{>n})\|_{L_{x,y}^2}\\
				&+\|\tilde{P}_{n/10}(f_{n/2}k_{>n})\|_{L_{x,y}^2}.
			\end{aligned}
		\end{equation}

		Firstly, by the monotone convergence theorem, we have 
		$$
		\lim_{n\uparrow\infty} \left\|\tilde{P}_{>n/10}(fk)\right\|_{L_{x,y}^2}=0
		$$
		and
		$$
		\lim_{n\uparrow\infty} \left\|\tilde{P}_{n/10}(f_{>n/2}k_{>n})\right\|_{L_{x,y}^2}\lesssim \lim_{n\uparrow\infty}\|f_{>n/2}\|_{L_x^2}\cdot \| k\|_{L_x^\infty L_y^2}=0.
		$$
		
		Secondly, it is clear that 
		$$
		\|\tilde{P}_{n/10}(f_{n/2}k_{>n})\|_{L_{x,y}^2}=0.
		$$
		
		Thus, according to \eqref{partition}, we only need to prove 
		\begin{equation*}
			\begin{aligned}
				\lim_{n\uparrow\infty}\|\tilde{P}_{>n/10}(fk_n)\|_{L_{x,y}^2}=0.
			\end{aligned}
		\end{equation*}
		Furthermore, since 
		$$ \lim_{n\uparrow\infty} \|\tilde{P}_{>n/10}(f_{>n/20}k_n)\|_{L_{x,y}^2}=0,\  \|\tilde{P}_{>n/10}(f_{n/20}k_{n/100})\|_{L_{x,y}^2}=0,
		$$ the above equality can be reduced to 
		\begin{equation}\label{auxi_lim_2}
			\begin{aligned}
				\lim_{n\uparrow\infty}\|\tilde{P}_{>n/10}(f_{n/20}k_{n/100<\cdot<n})\|_{L_{x,y}^2}=0.
			\end{aligned}
		\end{equation}
		
		By the fact $\{f:\hat{f}\in C_0^\infty\}$ is dense in $L^2_x$, \eqref{auxi_lim_2} can be deduced by 
		$$
		\lim_{n\uparrow\infty}\|fk_{n/100<\cdot<n}\|_{L_{x,y}^2}=0.
		$$
		
		What's more, because $C_0^\infty$ is dense in $L_x^2$, \eqref{auxi_lim_2} can be reduced to
		\begin{equation}\label{auxi_lim_3}
			\begin{aligned}
				\lim_{n\uparrow\infty}\|k_{n/100<\cdot<n}\|_{L_{[-c,c]}^2L_y^2}=0,
			\end{aligned}
		\end{equation}
		for some $c>0$.
		
		Because of $k\in L_x^\infty L_y^2$, $\chi_{[-2c,2c]}(x)k\in L_{x,y}^2$, we have 
		$$
		\lim_{n\uparrow\infty}\|\tilde{P}_{n/100<\cdot<n}(\chi_{[-2c,2c]}k)\|_{L_{x}^2L_y^2}=0.
		$$
		Thus, we only need to prove 
		\begin{equation}\label{auxi_lim_4}
			\lim_{n\uparrow\infty}\|\tilde{P}_{n/100<\cdot<n}(\chi_{[-2c,2c]^c}k)\|_{L_{[-c,c]}^2L_y^2}=0.
		\end{equation}

		Now, we illustrate \eqref{auxi_lim_4}. 
		There exists a Schwartz function $\varphi$ such that 
		$$
		\begin{aligned}
			\tilde{P}_{n/100<\cdot<n}(\chi_{[-2c,2c]^c}k)=&\int_{[-2c,2c]^c} n\varphi(n(x-z))k(z,y)dz\\
			\lesssim&\int_{[-2c,2c]^c} n|z|^{-l}n^{-l}|k(z,y)|dz,
		\end{aligned}
		$$
		which implies \eqref{auxi_lim_4}. Hence, we finish the proof of \eqref{convergence_result}.
	\end{proof}
	
	\begin{prop}\label{prop:u_R_L_x^2_priori_esti}
		Assume $u_0\in L_x^2$ and $ k\in L_x^\infty L_y^2$. For any fixed $ R,T>0$, let $u_R$ be the solution constructed in Proposition \ref{localwellres}. Then for any stopping time $0\leq \sigma\leq \sigma_R^T(u_R)$ almost surely, we have
		\begin{equation*}
			\left\|u_R\right\|_{L_\omega^pL^\infty_{\sigma}L_x^2}\leq C(p,T,\|k\|_{L_x^\infty L_y^2})\|u_0\|_{L_x^2}, \quad \forall~2\leq p<\infty.
		\end{equation*}
	\end{prop}
	\begin{proof}[\textbf{Proof}]		
		Recall that $\theta\in C_0^\infty (\mathbb{R})$ and $\mathrm{supp}~\theta\subset [-2,2]$, $\theta|_{[-1,1]} = 1$. We set
		\begin{equation*}
			k_m(x,y):=\mathscr{F}_{\xi}^{-1}(\theta(\xi/m)\mathscr{F}_x(k(x,y))(\xi)),~ u_0^{(m)}:=\mathscr{F}_{\xi}^{-1}(\theta(\xi/m)\mathscr{F}_x(u_0)(\xi))			
		\end{equation*}
		and 
		\begin{equation*}
			\Phi_mf(x):=\int_{\mathbb{R}} k_m(x,y) f(y)~dy, \quad \forall~ f\in L^2(\mathbb{R}).
		\end{equation*}
		Let $u_R^{(m)}$ be the solution of \eqref{truncated_eq} constructed in Proposition \ref{localwellres} with initial data $u_0^m$ and noise kernel $k_m$. By Young's inequality, it is clear that for any $s\geq 0$ we have 
		$\|u_0^{(m)}\|_{H^s}\leq C_s\langle m\rangle^s \|u_0\|_{L^2}$
		and
		\begin{align*}
			\left\|J_x^sk_m\right\|_{L_x^\infty L_y^2} & = \|\mathscr{F}_{\xi}^{-1}(\langle\xi\rangle^s\theta(\xi/m)\mathscr{F}_x(k(x,y))(\xi))\|_{L_x^\infty L_y^2} \\
			& \leq \| \mathscr{F}_{\xi}^{-1}(\langle\xi\rangle^s\theta(\xi/m))\|_{L_x^1} \|k(x,y) \|_{L_x^\infty L_y^2} \\
			& \leq C_s\langle  m\rangle^s\|k(x,y) \|_{L_x^\infty L_y^2}.
		\end{align*}
		Note that $\|u_0^{(m)}\|_{L^2}$, $\|k_m\|_{L_x^\infty L_y^2}$ are uniformly bounded for $m$. Therefore, according to Proposition \ref{localwellres} and Lemma \ref{lem:regularity_priori_estimate}, we have
		\begin{equation}\label{priori_approximation_norm}
			\|u^{(m)}_R\|_{L^p_\omega L_{\sigma\wedge\sigma_{R}^T(u_R^{(m)})}^\infty L^2_x}\leq C(p,T,\|k\|_{L_x^\infty L_y^2})\|u_0\|_{L^2_x}.
		\end{equation}		
		Now, we turn to the convergence property of $\{u_R^{(m)}\}$. For any $0<T_1<\min\{1,T\}$, by the proof of the local well-posedness and Lemma \ref{stochasticesti} one has
		\begin{align*}
			& \quad \|u_R-u_R^{(m)}\|_{L_\omega^2Z^{0,b}_{T_1}}\\
			& \lesssim \|u_0-u_0^{(m)}\|_{L^2}+ \left\|\int_{0}^{t} U(t-t')(B_R(u_R)-B_R(u^{(m)}_R))_{x}~dt'\right\|_{L_\omega^2Z^{0,b}_{T_1}}\\
			& \quad + \left\| \int_{0}^{t} U(t-t')((u_R-u_R^{(m)})\Phi_m dW_{t'})\right\|_{L_\omega^2Z^{0,b}_{T_1}}\\
			& \quad + \left\| \int_{0}^{t} U(t-t')(u_R(\Phi-\Phi_m)dW_{t'})\right\|_{L^2_\omega Z^{0,b}_{T_1}}\\
			& \lesssim \|u_0-u_0^{(m)}\|_{L_x^2}+T_1^{\varepsilon}R\|u_R-u_R^{(m)}\|_{L_\omega^2Z^{0,b}_{T_1}} +T_1^{1/2}\|u_R-u_R^{(m)}\|_{L^2_\omega Z^{0,b}_{T_1}}\\
			& \quad +\|u_R(\Phi-\Phi_m)e_j\|_{L^2_\omega L^2_{T_1}l^2_jL^2_x}.
		\end{align*}
		Therefore, if we choose $T_1=T_1(R,\varepsilon,T,\|k\|_{L_x^\infty L_y^2})$ sufficiently small, then we have 
		\begin{equation*}
			\|u_R-u_R^{(m)}\|_{L_\omega^2Z^{0,b}_{T_1}}\lesssim \|u_0-u_0^{(m)}\|_{L_x^2}+\|u_R(\Phi-\Phi_m)e_j\|_{L^2_\omega L^2_{T_1}l^2_jL^2_x}.
		\end{equation*}
		According to Lemma \ref{lem:convergence_need_to_explain}, the right hand converges to zero as $m\rightarrow\infty$. Thus one has $\lim_{m\rightarrow \infty}\|u_R-u_R^{(m)}\|_{L_\omega^2Z^{0,b}_{T_1}}=0$. Then, by dividing $[0,T]$ into finite intervals, we have
		\begin{equation}\label{L_omega^2_convergence}
			\lim_{m\rightarrow \infty}\|u_R-u_R^{(m)}\|_{L_\omega^2Z^{0,b}_{T}}=0.
		\end{equation}
		By \eqref{L_omega^2_convergence}, $\{u_R^{(m)}\}_{m\in\mathbb{N}^+}$ converges to $u_R$ in $\mathbb{P}$ sense. Therefore, for any $0<\delta<R$, $\epsilon>0$ there exists a sequence $\{m_n\}$ such that
		\begin{equation}\label{converge_in_the_sense_of_probabilty}
			\mathbb{P}( \|u_R-u_R^{(m_n)} \|_{Z_T^{0,b}}>\delta)<\frac{\epsilon}{2^n}.
		\end{equation}
		Let $E_{\delta,\epsilon}:=\bigcap_n\{ \omega\in\Omega: \|u_R-u_R^{(m_n)} \|_{Z_T^{0,b}}\leq \delta\}$. We have
		$$\mathbb{P}(E_{\delta,\epsilon}^c)\leq \sum_n\mathbb{P}( \|u_R-u_R^{(m_n)} \|_{Z_T^{0,b}}>\delta)<\epsilon.$$ 
		Note that for any $\omega\in E_{\delta,\epsilon}$, one has $\sigma^T_{R-\delta}(u_R)\leq \sigma^T_{R}(u_R^{(m_n)})$, $\forall~n$ and $\|u_R\|_{L^\infty_TL^2_x}\leq \|u_R^{(m_n)}\|_{L^\infty_TL_x^2}+\delta$ . Let $\sigma_\delta = \sigma\wedge \sigma^T_{R-\delta}(u_R)$. Then
		$$
		\|\chi_{E_{\delta,\epsilon}} u_R \|_{L^p_{\omega}L^\infty_{\sigma_\delta}L_x^2}\leq  \|\|u_R^{(m_n)}\|_{L^\infty_{\sigma_\delta}L_x^2}+\delta\|_{L^2_{\omega}}\leq  \delta+\|u^{(m_n)}_R\|_{L^p_\omega L_{\sigma_{R}^T(u_R^{(m_n)})}^\infty L^2_x}.
		$$
		By \eqref{priori_approximation_norm} one has
		\begin{align*}
			\|\chi_{E_{\delta,\epsilon}} u_R \|_{L^p_{\omega}L^\infty_{\sigma_\delta}L_x^2}\leq  \delta+C(p,T,\|k\|_{L_x^\infty L_y^2})\|u_0\|_{L^2_x}.
		\end{align*}
		By taking $\epsilon\rightarrow 0^+$ we obtain $\|u_R \|_{L^p_{\omega}L^\infty_{\sigma_\delta}L_x^2}\leq  \delta+C(p,T,\|k\|_{L_x^\infty L_y^2})\|u_0\|_{L^2_x}$. Since $\lim_{\delta\rightarrow 0^+}\sigma_{R-\delta}^T(u_R) = \sigma_R^T(u_R)$, we conclude the proof by taking $\delta$ tend to zero.
	\end{proof}

	\begin{prop}\label{prop:global_well-posedness}
		Let $k\in L_x^\infty L_y^2$, $u_0\in L^2_x$, $u$ be the solution of \eqref{mskdv}. For any $T>0$, we have $\sigma^T = T$, $\text{a.s.}$  $\mathbb{P}$.
	\end{prop}
	\begin{proof}[\textbf{Proof}]
		We argue by contradiction. If not, there exist $T,\mu>0$ and $E:=\{\omega:\sigma^T<T\}\in \mathcal{F}_T$ such that $\mathbb{P}(E)=\mu>0$.
		
		For brevity, we denote $\sigma_R^T(u_R)$ by $\sigma_R^T$. According to $\sigma^T = \lim_{R\rightarrow \infty} \sigma_R^T$  $\mathrm{a.s.}$ $\mathbb{P}$, thus for any $\delta>0$ there exists $R = R(\delta)>0$ such that 
		$\mathbb{P}(\{\sigma^T-\sigma^T_R>\delta\})<\mu/2$.
		
		We consider the equation 
		\begin{equation*}
			\begin{aligned}
				v(t) = U(t)u(\sigma_R^T)-\int_{0}^{t} U(t-t')\left(B_{\tilde{R}}(v)\right)_x~dt'+\int_{0}^{t}U(t-t') (v\Phi dW_{t'+\sigma_R^T}),
			\end{aligned}
		\end{equation*}
		where $v(t)$ is $\mathcal{F}_{t+\sigma_R^T}$ adapted. Here, $W_{t+\sigma_R^T}-W_{\sigma_R^T}$ is also a Brownian motion adapted to $\{\mathcal{F}_{t+\sigma_R^T}\}_{t\geq 0}$ according to the strong Markov property of the Brownian motion. See for example \cite{karatzas2014brownian}.
		
		By Lemmas \ref{lowfrebi}, \ref{bilinearestforkdv}, \ref{lowfrebi}, Propositions \ref{stochesti}, \ref{prop:u_R_L_x^2_priori_esti}, there exists $ \tilde{R}$  and $\tilde{T} = \tilde{T}(\tilde{R})$ sufficiently small such that 
		$$\|v\|_{L^2_\omega Z_{\tilde{T}}^{0,b}}\leq C\|u(\sigma_R^T)\|_{L^2_\omega L^2_x}\leq C(\|k\|_{L_x^\infty L_y^2},T,\|u_0\|_{L_x^2}).$$
		Then, by Chebyshev's inequality we have
		$\mathbb{P}(\{\|v\|_{Z_{\tilde{T}}^{0,b}}<\tilde{R}\})\geq 1- C^2/\tilde{R}^2$. Thus, we can choose $\tilde{R}$ sufficiently large such that $C^2/\tilde{R}^2<\mu/3$. 
		Therefore, we have
		\begin{align*}
			\mathbb{P}(\{ \sigma_R^T+\tilde{T}\leq \sigma^T\})& \geq \mathbb{P}(\{\|v\|_{Z_{\tilde{T}}^{0,b}}<\tilde{R}\}\cap E)\\
			& \geq 1- (\mathbb{P}(E^c)+\mathbb{P}(\{\|v\|_{Z_{\tilde{T}}^{0,b}}<\tilde{R}\}^c))\\
			& \geq \frac{2}{3}\mu,
		\end{align*}
		By choosing $\delta = \tilde{T}/2$, we conclude the proof by contradiction.
	\end{proof}
	
	Now we show the high regularity of the solution if $u_0$ and $k$ have high regularity.
	
	\begin{prop}\label{highreglobal}
		Suppose that $s\geq 
		0$, $J^s_xk\in L_x^\infty L_y^2$, $u_0\in H^s_x$. Let $u$ be the unique global solution of \eqref{mskdv} with $u\in Z^{0,b}_T$ for any $T>0$, $\mathrm{a.s.}$ $\mathbb{P}$. Then,  $u$ is also in  $Z^{s,b}_T$ almost surely.
	\end{prop}
	\begin{proof}[\textbf{Proof}]
		Let $u_R$ be the solution of \eqref{truncated_eq} constructed in Proposition \ref{localwellres}. 
		
		On the one hand, by the fact $T = \lim_{R\rightarrow \infty} \sigma^T_R(u_R)$, $ \mathrm{a.s.}$ $\mathbb{P}$, for any $\epsilon>0$, there exists $R = R(\epsilon)$ such that $\mathbb{P}(\{\sigma^T_{R}(u_R)>T/2\})>1-\epsilon$. 
		
		On the other hand, by Proposition \ref{localwellres} and Chebyshev's inequality one has
		$$\mathbb{P}(\{\|u_R\|_{Z^{s,b}_T}<M\})\geq 1- \frac{C(\|u_0\|_{H^s}, \|J^s_xk\|_{L_x^\infty L_y^2},R,T)^2}{M^2}.$$
		Thus by choosing $M = M(R)$ sufficiently large, we have $\mathbb{P}(\{\|u_R\|_{Z^{s,b}_T}<M\})> 1- \epsilon$. Then
		$$
		\mathbb{P}(\{\|u\|_{Z^{s,b}_{T/2}}<M\})>\mathbb{P}(\{\sigma^T_{R}(u_R)>T/2\}\cap \{\|u_R\|_{Z^{s,b}_T}<M\})>1-2\epsilon.
		$$
		Hence, by taking $\epsilon$ tend to zero, for any $T>0$ we have $\mathbb{P}(\{\|u\|_{Z^{s,b}_{T/2}}<\infty\}) = 1$, which concludes the proof.
	\end{proof}
	
	\appendix
	\section{The proof of Lemma \ref{truncated_equation}}\label{truncated}
	For the sake of simplicity, we may assume $R = 1$. Let 
	$$f(t,\xi) = \mathscr{F}_x(U(-t)u(t))(\xi),\ g(t,\xi) = \mathscr{F}_x(U(-t)v(t))(\xi),$$
	$$\theta_u(t) = \theta(\|u\|^2_{X^{0,b}_t}),\ \theta_v(t) = \theta(\|v\|^2_{X^{0,b}_t}).$$
	We define
	\begin{equation*}
		\sigma_u =\left\{
		\begin{aligned}
			&\inf\{t> 0: \|u\|_{X^{0,b}_t}\geq 2\}, &\|u\|_{X^{0,b}_T}\geq 2,\\
			&T, & \|u\|_{X^{0,b}_T} < 2.
		\end{aligned}
		\right.
	\end{equation*}
	Similarly, we define $\sigma_v$.%
	
	Firstly, we prove $\|\theta(\|u\|^2_{X_t^{0,b}})u\|_{X_T^{0,b}}\lesssim 1$. According to Lemma \ref{lem:equa-norm} and the definition of $X^{0,b}_T$, one has $\|\theta(\|u\|^2_{X_t^{0,b}} )u\|_{X^{0,b}_T}\sim 
	\|\theta(\|u\|^2_{X_t^{0,b}})u\|_{X^{0,b}_{\sigma_u}}$ and
	\begin{align*}
		&\quad \|\theta(\|u\|^2_{X_t^{0,b}} )u\|_{X^{0,b}_{\sigma_u}}^2\\
		&\lesssim \int_{\mathbb{R}}\int_0^{\sigma_u} (1+t^{-2b})|\theta(\|u\|^2_{X_t^{0,b}})f(t,\xi)|^2~dtd\xi\\
		&\quad+\int_{\mathbb{R}}\int_0^{\sigma_u}\int_0^t\frac{|\theta(\|u\|^2_{X_t^{0,b}})f(t,\xi)-\theta(\|u\|^2_{X_{t'}^{0,b}})f(t',\xi)|^2}{|t-t'|^{1+2b}}~dt'dtd\xi\\
		& \lesssim \|u\|^2_{X_{\sigma_u}^{0,b}}+\int_{\mathbb{R}}\int_0^{\sigma_u}\int_0^t\frac{|\theta(\|u\|^2_{X_t^{0,b}})-\theta(\|u\|^2_{X_{t'}^{0,b}})|^2|f(t',\xi)|^2}{|t-t'|^{1+2b}}~dt'dtd\xi.
	\end{align*}
	Since $0\leq \theta\leq 1$, $|\theta'|\lesssim 1$, $t'\leq t$, one has
	\begin{align*}
		&\quad |\theta(\|u\|^2_{X_t^{0,b}})-\theta(\|u\|^2_{X_{t'}^{0,b}})|\\
		&\lesssim \|u\|_{X_t^{0,b}}^2-\|u\|^2_{X_{t'}^{0,b}}\\
		& \sim \int_{\mathbb{R}}\int_{t'}^{t} \left((1+t_1^{-2b})|f(t_1,\eta)|^2+\int_0^{t_1}\frac{|f(t_1,\eta)-f(t_2,\eta)|^2}{|t_1-t_2|^{1+2b}}dt_2\right)~dt_1d\eta.
	\end{align*}
	Then by integrating $t$ and the Hardy inequality, one has
	\begin{align*}
		&\quad\int_{\mathbb{R}}\int_0^{\sigma_u}\int_0^t\frac{|\theta(\|u\|^2_{X_t^{0,b}})-\theta(\|u\|^2_{X_{t'}^{0,b}})|^2|f(t',\xi)|^2}{|t-t'|^{1+2b}}~dt'dtd\xi\\
		&\lesssim \int_{\mathbb{R}^2_{\xi,\eta}}\int_0^{\sigma_u}\int_{t'}^{\sigma_u}\frac{|f(t',\xi)|^2}{|t_1-t'|^{2b}} \bigg((1+t_1^{-2b})|f(t_1,\eta)|^2\\
		&\hspace{150pt}+\int_0^{t_1}\frac{|f(t_1,\eta)-f(t_2,\eta)|^2}{|t_1-t_2|^{1+2b}}~dt_2\bigg)~dt_1  dt' d\eta d\xi\\
		&\lesssim \|u\|_{X_{\sigma_u}^{0,b}}^4.
	\end{align*}
	Thus
	\begin{equation}\label{geconproc}
		\|\theta(\|u\|^2_{X_t^{0,b}} )u\|_{X^{0,b}_{\sigma_u}} \lesssim \|u\|_{X^{0,b}_{\sigma_u}} + \|u\|^2_{X^{0,b}_{\sigma_u}} \lesssim 1.
	\end{equation}
	Similarly, for $s\geq 0$ we have
	\begin{align*}
		\|\theta(\|u\|^2_{X_t^{0,b}} )u\|_{X^{s,b}_{T}}\lesssim \|u\|_{X^{s,b}_{\sigma_u}}+\|u\|_{X^{0,b}_{\sigma_u}}\|u\|_{X^{s,b}_{\sigma_u}}\lesssim \|u\|_{X^{s,b}_T}.
	\end{align*}
	We omit the details.
	
	Now, we show $\|\theta(\|u\|^2_{X_t^{0,b}} )u-\theta(\|v\|^2_{X_t^{0,b}} )v\|_{X^{0,b}_T}\lesssim \|u-v\|_{X^{0,b}_T}$. Without loss of the generality, we assume $\sigma_v\leq \sigma_u$. Therefore, we have  $\|\theta(\|u\|^2_{X_t^{0,b}} )u-\theta(\|v\|^2_{X_t^{0,b}}\big )v\|_{X^{0,b}_T}\sim \|\theta(\|u\|^2_{X_t^{0,b}} )u-\theta(\|v\|^2_{X_t^{0,b}} )v\|_{X^{0,b}_{\sigma_u}}$. For the sake of simplicity, we would also assume $\sigma_u = T$. Then $\|u\|_{X^{0,b}_{T}}\leq 2$. If $\|v\|_{X^{0,b}_{T}}\geq 3$, we have
	\begin{align*}
		\|\theta(\|u\|^2_{X_t^{0,b}} )u-\theta(\|v\|^2_{X_t^{0,b}} )v\|_{X^{0,b}_T}&\leq \|\theta(\|u\|^2_{X_t^{0,b}} )u\|_{X^{0,b}_T}+\|\theta(\|v\|^2_{X_t^{0,b}} )v\|_{X^{0,b}_T}\\
		&\lesssim \|v\|_{X^{0,b}_T}\sim \|u-v\|_{X^{0,b}_T}.
	\end{align*}
	Thus we would assume $\|v\|_{X^{0,b}_T}<3$. By the triangle inequality,
	\begin{align*}
		\|\theta(\|u\|^2_{X_t^{0,b}} )u-\theta(\|v\|^2_{X_t^{0,b}} )v\|_{X^{0,b}_T}&\leq \|\theta(\|u\|^2_{X_t^{0,b}} )(u(t)-v(t))\|_{X^{0,b}_T}\\
		&\quad+\|(\theta(\|u\|^2_{X_t^{0,b}} )-\theta(\|v\|^2_{X_t^{0,b}} ))v(t)\|_{X^{0,b}_T}.
	\end{align*}
	By the same argument as in \eqref{geconproc}, we have $\|\theta(\|u\|^2_{X_t^{0,b}} )(u(t)-v(t))\|_{X^{0,b}_T}\lesssim \|u-v\|_{X^{0,b}_T}$.
	
	Since $|\theta(\|u\|^2_{X_t^{0,b}} )-\theta(\|v\|^2_{X_t^{0,b}} )|\lesssim \|u-v\|_{X_t^{0,b}}\lesssim \|u-v\|_{X_T^{0,b}}$, for the estimate of $\|(\theta(\|u\|^2_{X_t^{0,b}} )-\theta(\|v\|^2_{X_t^{0,b}}))v(t)\|_{X^{0,b}_T}$ we focus on the term
	\begin{align*}
		&\int_{\mathbb{R}}\int_0^{T}\int_0^t\frac{|\theta(\|u\|^2_{X_t^{0,b}})-\theta(\|v\|^2_{X_t^{0,b}})-\theta(\|u\|^2_{X_{t'}^{0,b}})+\theta(\|v\|^2_{X_{t'}^{0,b}})|^2|g(t',\xi)|^2}{|t-t'|^{1+2b}}~dt'dtd\xi
	\end{align*}
	Let $h(t) = \|u\|_{X_t^{0,b}}^2-\|v\|_{X^{0,b}_t}^2$, $\forall~0<t<T$. Then
	\begin{align*}
		&\quad\theta\big(\big\|u\|^2_{X_t^{0,b}}\big)-\theta\big(\big\|v\|^2_{X_t^{0,b}}\big)-\theta\big(\big\|u\|^2_{X_{t'}^{0,b}}\big)+\theta\big(\big\|u\|^2_{X_{t'}^{0,b}}\big)\\
		& = \int_0^1 \theta'(\|v\|_{X^{0,b}_t}^2+\delta h(t))h(t)-\theta'(\|v\|_{X^{0,b}_{t'}}^2+\delta h(t'))h(t')~d\delta.
	\end{align*}
	Since $\theta\in C_0^\infty(\mathbb{R})$, we have
	\begin{align*}
		&\quad|\theta(\|u\|^2_{X_t^{0,b}})-\theta(\|v\|^2_{X_t^{0,b}})-\theta(\|u\|^2_{X_{t'}^{0,b}})+\theta(\|u\|^2_{X_{t'}^{0,b}})|^2\\
		&\lesssim |h(t)-h(t')|^2+ |h(t')|^2\int_0^1 |\|v\|_{X^{0,b}_t}^2+\delta h(t)-(\|v\|_{X^{0,b}_{t'}}^2+\delta h(t'))|~d\delta\\
		&\lesssim |h(t)-h(t')|^2+ |h(t')|^2(\|u\|_{X^{0,b}_t}^2-\|u\|_{X^{0,b}_{t'}}^2)+ |h(t')|^2(\|v\|_{X^{0,b}_t}^2-\|v\|_{X^{0,b}_{t'}}^2).
	\end{align*}
	Note that $|h(t')|\leq \|u-v\|_{X_{t'}^{0,b}}(\|u\|_{X_{t'}^{0,b}}+\|v\|_{X_{t'}^{0,b}})\lesssim\|u-v\|_{X^{0,b}_T}$. Thus, by the same argument as in \eqref{geconproc} we have
	\begin{align*}
		&\quad\int_{\mathbb{R}}\int_0^{T}\int_0^t\frac{|h(t')|^2(\|u\|_{X^{0,b}_t}^2-\|u\|_{X^{0,b}_{t'}}^2)|g(t',\xi)|^2}{|t-t'|^{1+2b}}~dt'dtd\xi\\
		&\lesssim \|u-v\|_{X^{0,b}_T}^2\|u\|_{X_{T}^{0,b}}^2\|v\|_{X^{0,b}_T}^2\lesssim \|u-v\|_{X^{0,b}_T}^2
	\end{align*}
	and
	\begin{align*}
		&\quad\int_{\mathbb{R}}\int_0^{T}\int_0^t\frac{|h(t')|^2(\|v\|_{X^{0,b}_t}^2-\|v\|_{X^{0,b}_{t'}}^2)|g(t',\xi)|^2}{|t-t'|^{1+2b}}~dt'dtd\xi\\
		&\lesssim \|u-v\|_{X^{0,b}_T}^2\|v\|_{X^{0,b}_T}^4\lesssim \|u-v\|_{X^{0,b}_T}^2.
	\end{align*}
	Since
	\begin{align*}
		&\quad|h(t)-h(t')| \\
		&\leq \int_{\mathbb{R}}\int_{t'}^{t} (1+t_1^{-2b})||f(t_1,\eta)|^2-|g(t_1,\eta)|^2|~dt_1\\
		&\quad+\int_{\mathbb{R}}\int_{t'}^{t}\int_0^{t_1}\frac{||f(t_1,\eta)-f(t_2,\eta)|^2-|g(t_1,\eta)-g(t_2,\eta)|^2|}{|t_1-t_2|^{1+2b}}~dt_2dt_1d\eta,
	\end{align*}
	by the Cauchy--Schwarz inequality we have
	\begin{align*}
		&\quad|h(t)-h(t')|^2\\
		&\lesssim (\|u\|_{X_t^{0,b}}+\|v\|_{X_t^{0,b}})^2 \bigg(\int_{\mathbb{R}}\int_{t'}^{t} (1+t_1^{-2b})|f(t_1,\eta)-g(t_1,\eta)|^2~dt_1\\
		&\quad+\int_{\mathbb{R}}\int_{t'}^{t}\int_0^{t_1}\frac{|f(t_1,\eta)-g(t_1,\eta)-f(t_2,\eta)+g(t_2,\eta)|^2}{|t_1-t_2|^{1+2b}}~dt_2dt_1d\eta\bigg).
	\end{align*}
	Since $\|u\|_{X_t^{0,b}}+\|v\|_{X_t^{0,b}}\lesssim 1$, we conclude the proof by integrating $t$ first and using the Hardy inequality.
	
	\section{The proof of Lemma \ref{fractionleib}}\label{proofoffrac}
	The case $s = 0$ is trivial. Thus we would assume $s>0$. In fact this is the result of the fractional Leibniz rule. See for example \cite{muscalushlag,benea2016multiple,benea2017quasi}. We give the proof here for completeness. With a little abuse of notation, we use the classical inhomogeneous Littlewood--Paley decomposition here. Let $\varphi\in C_0^\infty((-5/4,5/4))$,  $\psi(\xi) = \varphi(\xi)-\varphi(2\xi)$, $\varphi|_{[-1,1]} = 1$. We define $P_N = \mathscr{F}^{-1}\psi(\cdot/N)\mathscr{F}, N\geq 2$; $P_1 = \mathscr{F}^{-1}\varphi(\cdot)\mathscr{F}$. By the Bony paraproduct decomposition one has
	\begin{align*}
		\|u(x)k(x,y)\|_{L^2_yH_x^s}&\lesssim \|N^s P_N u(x)P_{\lesssim N}k(x,y)\|_{l^2_{N\geq 1}L_{x,y}^2}\\
		&\quad+\sum_{N\gg 1}\|J^s_x(P_Nu(x)P_{\gg N}k(x,y))\|_{L^2_{x,y}}.
	\end{align*}
	Thus we only need to show
	$$\|P_N u(x)P_{\lesssim N}k(x,y)\|_{L_{x,y}^2}\lesssim \|u\|_{L^2_x}\|k\|_{L_x^\infty L_y^2},\quad N\geq 1$$
	and
	$$\|J^s_x(P_Nu(x)P_{\gg N}k(x,y))\|_{L^2_{x,y}}\lesssim \|u\|_{L^2_x}\|J^s_xk\|_{L_x^\infty L_y^2},\quad N\geq 1.$$
	Since $P_{\lesssim N} f = \varphi_{CN}*f$, by H\"{o}lder and Young inequalities, we have
	\begin{align*}
		\|P_N u(x)P_{\lesssim N}k(x,y)\|_{L_{x,y}^2}&\lesssim \|P_Nu\|_{L^2_x}\||\varphi_{CN}|*\|k(x,\cdot)\|_{L^2_y}\|_{L^\infty_x}\\
		&\lesssim \|u\|_{L^2_x}\|k\|_{L_x^\infty L_y^2}.
	\end{align*}
	By Lemma 2 on page 133 in \cite{stein1970singular} we only need to show
	$$\|D^s_x(P_Nu(x)D_{x}^{-s}P_{\gg N}k(x,y))\|_{L^2_{x,y}}\lesssim \|u\|_{L^2_x}\|k\|_{L_x^\infty L_y^2},\quad N\geq 1.$$
	By scaling we reduce the inequality to $N = 1$. Let $\tilde{\psi}(\xi) = \psi(\xi/C)$. Then
	\begin{align*}
		&\quad D^s_x(P_1 u(x)D_{x}^{-s}P_{\gg 1}k(x,y)) \\
		& = \frac{1}{2\pi}\int_{\mathbb{R}^2} |\xi|^se^{ix\xi}\hat{u}(\xi-\eta)\psi(\xi-\eta)|\eta|^{-s}(1-\tilde{\psi}(\eta))\mathscr{F}_xk(\eta,y)~d\xi d\eta.
	\end{align*}
	By the Taylor expansion, on the support of $\psi(\xi-\eta)(1-\tilde{\psi}(\eta))$ one has
	\begin{align*}
		\xi^s = \sum_{n = 0}^\infty a_{n,s}\eta^{s-n}(\xi-\eta)^n,\quad \xi>0;\quad |a_{n,s}|\leq C_s^n
	\end{align*}
	Thus we have
	$$P_{>0}D^s_x(P_1 u(x)D_{x}^{-s}P_{\gg 1}k(x,y)) = \sum_n a_{n,s}P_{>0}(D_x^nP_1 u(x) D_x^{-n}P_{\gg 1}k(x,y)).$$
	Then by H\"{o}lder and Young inequalities, for $C>2C_s$ we have
	\begin{align*}
		&\quad\|P_{>0}D^s_x(P_1u(x)D_{x}^{-s}P_{\gg 1}k(x,y))\|_{L^2_{x,y}}\\
		& \leq \sum_{n=0}^\infty |a_{n,s}|\|D_x^nP_1 u\|_{L^2_x} \|D_x^{-n}P_{\gg 1}k(x,y)\|_{L_x^\infty L_y^2}\\
		&\lesssim \sum_{n=0}^\infty  C_s^n2^n C^{-n} \|u\|_{L^2_x}\|k\|_{L_x^\infty L_y^2}\lesssim \|u\|_{L^2_x}\|k\|_{L_x^\infty L_y^2}.
	\end{align*}
	The argument for $\xi<0$ is similar. We conclude the proof.
	
	\section{Multilinear estimates for the KdV equation in Bourgain spaces}\label{recallmu}
	In this appendix, we recall the multilinear estimates for the KdV equation in the Bourgain spaces. We restate Lemma 5.7 in \cite{WangHHG} using the notation of the present paper. This theorem is used repeatedly in the proof of Lemma \ref{bilinearestforkdv} in this paper.
	
	Recall the Littlewood--Paley projection operators defined by \eqref{freque}, \eqref{modulationdeco}. Let $a_{\min}$, $a_{\mathrm{med}}$, and $a_{\max}$ denote the minimum, median, and maximum of $a_1$, $a_2$, $a_3$.
	\begin{lemma}[Lemma 5.7 in \cite{WangHHG}]\label{multiesti}
		Let $u,v,w\in L^2(\mathbb{R}^2)$ with $\|u\|_{L^2} = \|v\|_{L^2} = \|w\|_{L^2} = 1$. Assume $N_{\mathrm{min}}\geq 2$. Then
		\begin{itemize}
			\item[(a)] 
			\begin{align*}
				\left|\int_{\mathbb{R}^2}Q_{L_1}P_{N_1}uQ_{L_2}P_{N_2}v Q_{L_3}P_{N_3}w~dxdt\right| \lesssim L_{\mathrm{min}}^{1/2}N_{\mathrm{min}}^{1/2}.
			\end{align*}
			\item[(b)] If $N_{\mathrm{min}}\ll N_{\mathrm{max}}$, then
			\begin{align*}
				&\quad \left|\int_{\mathbb{R}^2}Q_{L_1}P_{N_1}uQ_{L_2}P_{N_2}v Q_{L_3}P_{N_3}w~dxdt\right|\\
				&\lesssim (L_1L_2L_3)^{1/2}N_{\mathrm{max}}^{-1/2}\left(\max_{1\leq j\leq 3}\{N_j L_j\}\right)^{-1/2}.
			\end{align*}
			\item[(c)] If $N_{\mathrm{min}}\sim N_{\mathrm{max}}$, then
			\begin{align*}
				\left|\int_{\mathbb{R}^2}Q_{L_1}P_{N_1}uQ_{L_2}P_{N_2}v Q_{L_3}P_{N_3}w~dxdt\right| \lesssim L_{\min}^{1/2}L_{\mathrm{med}}^{1/4}N_{\max}^{-1/4}.
			\end{align*}
		\end{itemize}
	\end{lemma}

	\section*{Acknowledgment}
	Jie Chen acknowledges the support of NSFC grants 12301116, 12171007. The authors thank Professors Boling Guo and Baoxiang Wang for their invaluable support and encouragement, also thank Yingzhe Ban and Ying Zhang for helping us find the literature on intrinsic characterizations of Sobolev spaces. The authors thank the referees for their suggestions and their careful reading of the manuscript.

	\section*{Conflict of interest statement}
	The authors have no conflict of interest to declare that are relevant to the content of this
	article.
	
	\section*{Data Availability Statement} Not applicable.

	\phantomsection

	\vspace{10pt}
	
	\begin{itemize}[leftmargin=5pt]
		\item[] \scriptsize\textsc{Jie Chen: School of Science, Jimei University, Xiamen 361021, P.R. China}
		
		\textit{E-mail address}: \textbf{jiechern@jmu.edu.cn}
		
		\item[] \scriptsize\textsc{Fan Gu: School of Statistics and Mathematics, Central University of Finance and Economics, Beijing 102206, P.R. China}
		
		\textit{E-mail address}: \textbf{gufan@amss.ac.cn}
	\end{itemize}
\end{document}